\title{Functional Stochastic Differential Equations with Positivity Constraints Driven by Fractional Brownian Motion}

\author{ Chadad Monir\thanks{Mathematics Department, Faculty of Sciences Semalalia, Cadi Ayyad University, Boulevard Prince Moulay Abdellah, P. O. Box 2390, Marrakesh 40000, Morocco.
		E-mail: m.chddad@gmail.com}
}

\documentclass[12pt,a4paper,openright]{article}
\usepackage[left = 2.5cm, right = 2.5cm, top = 3cm, bottom = 3cm]{geometry}
\usepackage{graphicx}
\usepackage[english]{babel}
\usepackage[utf8]{inputenc}
\usepackage[T1]{fontenc}
\usepackage{mathptmx}
\usepackage{ragged2e}
\usepackage{theorem}
\usepackage[colorlinks=true, citecolor=black, linkcolor=black, urlcolor=black]{hyperref}
\usepackage[colorlinks=true]{hyperref}
\usepackage{color}
\usepackage{amsfonts}
\hypersetup{colorlinks = true, linkcolor = blue, citecolor= red}
\usepackage{amsmath}
\usepackage{amsmath,amssymb} %Thick police
\usepackage{mathtools} %retour à la ligne dans les equations
\newcommand{\dint}{\displaystyle\int}

\newtheorem{theorem}{Theorem}[section]
\newtheorem{lemma}{Lemma}[section]

\newtheorem{examples}{Examples}[section]
\newtheorem{proposition}{Proposition}[section]

\newenvironment{proof}{{\sc Proof.}}{}
\newenvironment{keywords}{}{}

\begin{document}
	\date{}
	\newpage
	\maketitle
	%%%%%%%%%%%%%%%%%%%%%%%%%%%
	% abstract, keywords and Subject classification are optional.
	%%%%%%%%%%%%%%%%%%%%%%%%%%%
	\begin{abstract}
		This paper studies a stochastic functional differential equation driven by a fractional Brownian motion with Hurst parameter H>1/2, constrained to be reflected at 0. We prove the existence of solutions using the Euler method. However, uniqueness is demonstrated only under the condition that the fractional term exhibits constant argument deviation. Additionally, we establish the convergence of the method.
	\end{abstract}

	
	\section{Introduction}
				\hspace{0.5cm} In recent years, the field of functional differential equations (FDEs) has gained significant attention from mathematicians and researchers as a powerful mathematical tool for modeling dynamic systems that depend not only on the current state but also on the past history of the system. Unlike ordinary differential equations (ODEs), which involve only the current state and its derivatives, FDEs capture the memory or delay effects in the system's behavior. The concept of functional differential equations can be traced back to the early 20th century when mathematicians started exploring mathematical models that incorporate delay or history-dependent effects. The FDEs bridge the gap between ordinary differential equations and dynamic systems with historical dependencies, offering a versatile tool for understanding and predicting the behavior of complex systems in various scientific and engineering domains.\\
				
				Functional differential equations (FDEs) appear in many real-world scenarios and have numerous applications. For instance, in biology, they model population dynamics, where the growth rate depends not only on the current population size but also on its historical values. In economics, FDEs can be used to study economic systems with time lags in decision-making processes.  Functional differential equations can have various types of delays, such as constant delays, distributed delays, or variable delays depending on the function itself.\\
				
				In some applications, the quantities of interest are inherently positive. For instance, in finance, stock prices and other financial time series represent naturally positive values. Similarly, in ecology and epidemiology, population dynamics and the spread of infectious diseases are expressed as positive quantities.\\
				
				When constructing deterministic differential equation models to describe the delayed dynamics of such quantities, the dynamics may inherently maintain the positivity of these quantities or guide the trajectories back into the feasible domain with minimal effort upon reaching the boundary. This natural behavior motivates the consideration of Functional Differential Equations with positivity constraints. Equations with positivity constraints often lead us to contemplate the Skorokhod problem. Skorokhod, recognized as the pioneer of reflected stochastic differential equations, was the first to publish two seminal papers \cite{Sko, Sko1} titled 'Stochastic Equations for Diffusion Processes in a Bounded Region 1,2,' in which he demonstrated the existence of solutions for stochastic differential equations subject to a reflecting boundary condition. Specifically, Skorokhod proposed the identification of a pair of continuous, non-anticipating processes $(X(t),l(t)), t \geq 0$, such that
				\begin{equation}\label{First-SP}
					X(t)=x_{0}+\int_{0}^{t} a(s, X(s)) \,ds+\int_{0}^{t} b(s, X(s)) \,dW(s)+l(t), \quad \text { a.s.,}
				\end{equation}
				where $W$ is a standard Brownian motion, $X(t) \geq 0,\, t \geq 0$, $l$ is non-decreasing, $l(0)=0$ and
				$$
				l(t)=\int_{0}^{t} \mathbb{I}_{X(s)=0} \, dl(s), \quad t \geq 0 .
				$$
				The last requirement means also that $l$ may increase only when $X$ visits $0$ that is the corresponding measure $dl(s)$ is carried by the set $\{s: X(s)=0\}$. For this purpose, Skorokhod considered the first version of his problem (SP) formulated as follows : Let $f \in C([0,+\infty),\mathbb{R}), f(0) \geq 0$. A pair of continuous functions $(g,l)$ are called a solution of the one-dimensional Skorokhod Problem for $f$ if $g(0)=f(0) $, and if for all $t\geq 0$
				\begin{description}
					\item[1.] $g(t)=f(t)+l(t);$
					\item[2.] $g(t)\geq 0;$
					\item[3.] $l$ is non-decreasing on $[0, \infty)$, $l(0)=0$, and
					\item[4.] $l(t)=\dint_{0}^{t} \mathbb{I}_{g(s)=0} \,dl(s)$.	 
				\end{description}
				This pair is uniquely defined by
				\begin{equation}\label{SPsol}
					l(t)=\underset{0 \leq s \leq t}{\sup}\left(f(s)\right)^{-} \quad  \text{
						and} \quad g(t)=f(t)+\underset{0 \leq s \leq t}{\sup}\left(f(s)\right)^{-},\quad t \geq 0.\end{equation}
				The map from unconstrained to constrained function $f \rightarrow g$ is called the Skorokhod Map (SM). The SP has since proved to be convenient in analyzing a variety of processes and provided a very useful tool for the construction of many constrained processes. \\
				
				For stochastic differential equations driven by fractional Brownian motion (fBm for short) with a Hurst parameter $H>1/2$, and constrained by positivity, few results are available. Interested readers can refer to the works by Falkowski and Słomiński \cite{FS} and  Ferrante and Rovira \cite{FR}. Such equations are valuable tools for modeling systems that exhibit long memory and self-similarity. The inclusion of fractional Brownian motion endows the equations with long-range dependence and self-similarity properties, which are commonly observed in various real-world phenomena. The Hurst parameter, a critical factor with a value greater than $1/2$, determines the strength and persistence of memory within the process, significantly influencing the system's dynamics and behavior. In \cite{FS}, the authors prove an existence and uniqueness result for stochastic differential equations driven by fractional Brownian motion fBm with positivity constraints, using a pathwise approach based on the p-variation norm. Falkowski and Słomiński define the stochastic integral with respect to fBm as an extended Riemann-Stieltjes pathwise integral and use p-variation estimates. Unfortunately, in \cite{FR}, the approach used to prove the existence of a solution differs completely from that in \cite{FS}, as the integral with respect to fBm is defined as an extended Riemann-Stieltjes pathwise integral in the sense of Zähle \cite{16}, and the proofs and techniques rely on $\lambda$-Hölder norms. However, this approach only allows them to demonstrate the uniqueness of solutions over a small time interval. This limitation arises because, in general, the $\lambda$-Hölder norms of the difference between two regulator terms $\underset{0 \leq s \leq .}{\sup}\left(f^1(s)\right)^{-}$ and $\underset{0 \leq s \leq .}{\sup}\left(f^2(s)\right)^{-}$, cannot be controlled by the $\lambda$-Hölder norms of the difference of $f^1$ and $f^2$. \\
				
				Concerning equations driven by both fractional Brownian motion fBm and standard Brownian motion (known as mixed Stochastic differential equations) with reflecting boundary conditions, to the best of our knowledge, the first attempts to address such equations were made by Chadad and Erraoui in \cite{Cha-Err1, Cha-Err2}. The authors demonstrated the existence of a weak solution using a tightness argument. However, in \cite{Cha-Err1}, pathwise uniqueness was proved in the case of additive noise, while it remains an open problem in the multiplicative case. The difficulty arises from the simultaneous presence of two terms in the difference between two solutions, namely the integration with respect to fBm and the difference in the regulator terms, which cannot be controlled together by the same norm. This renders inapplicable the usual inequalities used for pathwise uniqueness, such as Gronwall, Bihari, and others. In \cite{Cha-Err2}, the authors succeeded in solving the problem of pathwise uniqueness, primarily because the fractional term acts with a delay relative to the other components of the equation. \\
				
				Similarly, for stochastic delay differential equations driven by fractional Brownian motion with a Hurst parameter $H > 1/2$ and constrained by positivity were first considered by Besalu and Rovira \cite{BR}. However, delay stochastic equations with positivity constraints driven by Brownian motion have been addressed in the work of Kinnally and Williams \cite{KW}. For equations without reflection driven by fractional Brownian motion, interested readers may refer to previous works (see e.g., \cite{BH, BO, Kub, NNT, NR}). \\
				
				As described earlier, incorporating delays and positivity constraints into equations driven by fractional Brownian motion offers a valuable framework for effectively modeling and comprehending complex systems characterized by long memory and self-similarity properties. This framework has served as a strong motivation for us to delve into the realm of functional stochastic differential equations (FDEs) while imposing positivity constraints, taking the form of	
				\begin{equation}
						\left\{ \begin{array}{lll}
							X(t) & = &\eta(0)+\dint_{0}^{t}b(s,X_{s})\,\mathrm{d}s+\dint_{0}^{t}\sigma(s,X_{s})\,\mathrm{d}B_{s}^{H}+Y(t) ,\,\, t\in[0,T], \\
							X(t) & = & \eta(t),\,\,t\in[-r,0].
						\end{array}\right. \label{equ}
				\end{equation}
				Where $B^{H}$ is a fractional Brownian motion with Hurst parameter $H \in (\frac{1}{2},1)$, defined in a complete probability space $\left(\Omega,\mathcal{F},\mathbb{P}\right)$ and $r>0$ the time delay. The initial condition $\eta$ is a $\theta$-H\"older function from $[-r,0]$ to $\mathbb{R_{+}}$ with $\theta > 1-H$. Let $\mathcal{C}_{r}:=C\left([-r,0],\mathbb{R}\right)$ be the space of $\mathbb{R}$-values continuous functions defined on the interval $[-r,0]$ endowed by the uniform norm $\left\Vert . \right\Vert_{\mathcal{C}_{r}}$. $X_{t}  \in \mathcal{C}_{r}$ denote the function defined by $X_{t}(s)=X(t+s), s\in [-r,0]$. While $Y$ is the regulator term which ensures the non-negativity of $X$. The integral with respect to $B^H$ is understood in the pathwise sense  of Zähle \cite{16} (A precise definition is given in Section 2).\\
				 
				 A particular case of the aforementioned equations was considered by Besalu and Rovira in \cite{BR}, where the coefficients take the form $\sigma(t,X_t)=\sigma(t,X(t-r))$ and $H>1/2$. The authors proved the existence and uniqueness of solutions using an induction argument. However, in our case, we face a challenge where the same approach cannot be applied. The usual proof relying on the Picard method of successive approximations is not applicable, mainly because the Skorokhod map lacks Lipschitz continuity in the $\lambda$-Hölder norm. Consequently, establishing a uniqueness result becomes unfeasible.\\

				To overcome this obstacle, we explore numerical methods for solving the equation. Among various approximation schemes, Euler's scheme stands out as the most commonly used. Several studies have been dedicated to analyzing and understanding Euler's scheme (see e.g., \cite{LL, LYZ}). Its widespread use can be attributed to its simplicity and the fact that the approximations generated by Euler's scheme themselves satisfy  stochastic differential equations. This is particularly significant, as it implies that any limit of a subsequence constitutes a solution to the given equation. Therefore, establishing equicontinuity for the Euler's approximations is sufficient to ensure the existence of a solution.\\
				
				The challenge of establishing the existence of a solution is tackled in two distinct steps. In the first step, we necessitate an estimate ~\eqref{est-frac} of the integral involving the fBm in \eqref{app sch}. To achieve this, we establish a practical estimate denoted as \eqref{est incre}. Subsequently, we demonstrate that the sequence of functions in \eqref{app sch} is both equicontinuous and bounded in $C([-r,T],\mathbb{R})$. In the second step, following the construction of a subsequence converging uniformly to a continuous function, we validate the convergence of all terms within the Euler scheme. In the second part of this paper we establish that \eqref{app sch} forms a Cauchy sequence in a particular case, thus confirming its convergence.\\
				
				The article is structured as follows. In Section 2, we provide fundamental information regarding integration with respect to fractional Brownian motion. We introduce appropriate normed spaces and state our main hypotheses. In Section 3, we define the Euler approximation for deterministic equations and establish that it forms a bounded and equicontinuous sequence. Finally, in Section 4, we demonstrate the existence of a solution.
				
			\section{Preliminaries}
				Now we present the functional framework required for solving Equation \eqref{equ}. We fix $\alpha \in (0,\frac{1}{2})$, $\mu \in (0,1]$ and we consider the following normed spaces. 
				
				\noindent 1. $C^{\mu} \left( [a,b]\right)  $ is the space of $\mu$-H\"older functions $f : [a,b] \longrightarrow \mathbb{R} $ equipped with the norm 
				\begin{equation}\label{mu-norm}
					\parallel f \parallel_{\mu,[a,b]}=\parallel f \parallel_{\infty}+\underset{=\mathfrak{H}_{\mu,[a,b]}(f)}{\underbrace{\underset{a\leq s<t\leq b}{\sup}\dfrac{\vert f(t)-f(s)\vert}{(t-s)^{\mu}}}}
					<+\infty,
				\end{equation}
				where	$$\parallel f \parallel_{\infty}=\underset{a \leq t \leq b}{\max}\vert f(t) \vert.$$
				2. 
				%$W_{0}^{\alpha,\infty}$ is the space of measurable functions $f : [0,T] \longrightarrow \mathbb{R}$ such that
				%	$$\parallel f \parallel_{\alpha,\infty}=\underset{0 \leq t \leq T}{\sup} \parallel f(t) \parallel_{\alpha}<+\infty,$$
				%	where 				 
				%	\begin{equation}
					%		\parallel f(t) \parallel_{\alpha}\, =\vert f(t) \vert + \int_{0}^{t}\dfrac{\vert f(t)-f(s) \vert}{(t-s)^{1+\alpha}}\, \mathrm{d}s.
					%	\end{equation}
				%	Given any $\epsilon$ such that $0<\epsilon < \alpha$, we have the following inclusions 
				%$$C^{\alpha+\epsilon} \subset W_{0}^{\alpha,\infty}\subset C^{\alpha-\epsilon}.$$
				%	In particluar, both the fractional Brownian motion $B^{H}$, with $H>\frac{1}{2}$, and the standard Brownian motion $W$ have their trajectories in $W_{0}^{\alpha,\infty}$. 
				We introduce the space $W_{0}^{1,\alpha,\infty}$ of measurable functions $f : [-r,T] \longrightarrow \mathbb{R}$ such that 
				\begin{equation}
					\parallel f \parallel_{\infty,\alpha,t}\, =  \left\Vert f \right\Vert_{\infty,t}+\int_{0}^{t}\, \left\Vert f(t-s+\cdot )-f(\cdot) \right\Vert_{\infty,s}(t-s)^{-\alpha-1}\, \mathrm{d}s<+\infty, \label{norm2}
				\end{equation}
				where
				$$\left\Vert f \right\Vert_{\infty,t}=\underset{-r \leq u \leq t}{\sup} \left| f(u)\right|.$$
				It is clear that $\parallel f \parallel_{\infty,\alpha,t}$, is non-decreasing in $t$ and it is one of strengths of this norm. 
				
				\noindent Given any $\epsilon >0$, we have the following inclusion 
				\begin{equation}\label{incl}
					C^{\alpha+\epsilon}([-r,T]) \subset W_{0}^{1,\alpha,\infty}.
				\end{equation}
				\noindent 3. $W_{[a,b]}^{1-\alpha,\infty}$ denotes the space of measurable functions $f : [a,b] \longrightarrow \mathbb{R}$ such that
				\begin{equation}
					\parallel f \parallel_{1-\alpha,\infty,[a,b]}= \underset{a \leq u <v \leq b}{\sup}\left( \dfrac{\vert f(v)-f(u) \vert}{( v-u )^{1-\alpha}} + \int_{u}^{v}\dfrac{\vert f(y)-f(u) \vert}{(y-u) ^{2-\alpha}}\, \mathrm{d}y\right)<+\infty.
				\end{equation}
				\noindent 4. Finally, $W_{[a,b]}^{\alpha,1}$ is the space of measurable functions $f : [a,b] \longrightarrow \mathbb{R}$ such that
				\begin{equation}
					\parallel f \parallel_{\alpha,1,[a,b]}:=\int_{a}^{b}\dfrac{\vert f(s)\vert}{(s-a)^{\alpha}}\, \mathrm{d}s+\int_{a}^{b}\int_{a}^{s}\dfrac{\vert f(s)-f(y) \vert}{(s-y) 	^{1+\alpha}}\,\mathrm{d}y\,\mathrm{d}s <+\infty.
				\end{equation}
				Now we give a brief review of the deterministic fractional calculus. The Weyl-Marchaud derivatives of $f:[a,b]\longrightarrow\mathbb{R}^{n}$ are given
				by:
				\[
				D_{a+}^{\alpha}f(x)=\dfrac{1}{\Gamma(1-\alpha)}\left(\dfrac{f(x)}{\left(x-a\right)^{\alpha}}+\alpha\int_{a}^{x}\dfrac{f(x)-f(y)}{\left(x-y\right)^{\alpha+1}}\, \mathrm{d}y\right)
				{1\!\!}1_{\left(a,b\right)}(x)
				\]
				and 
				\[
				D_{b-}^{\alpha}f(x)=\dfrac{\left(-1\right)^{\alpha}}{\Gamma(1-\alpha)}\left(\dfrac{f(x)}{\left(b-x\right)^{\alpha}}+\alpha\int_{x}^{b}\dfrac{f(x)-f(y)}{\left(y-x\right)^{\alpha+1}}\, \mathrm{d}y\right){1\!\!}1_{\left(a,b\right)}(x),
				\]
				where $\Gamma(\alpha) =\int_0^{\infty} t^{\alpha -1} e^{-t}dt$ is the Gamma function. Assuming that $D_{a+}^{\alpha}f\in L^{1}[a,b]$ and $D_{b-}^{1-\alpha}g_{b-}\in L^{\infty}[a,b]$,
				where $g_{b-}(x)=g(x)-g(b{-})$, the generalized (fractional) Lebesgue-Stieltjes integral of $f$ with respect to $g$, following the work of Z{\"a}hle \cite{16}, is defined as
				\begin{equation}
					\int_{a}^{b}f\, dg:=(-1)^{\alpha}\int_{a}^{b}\, D_{a+}^{\alpha}f(x)\, D_{b-}^{1-\alpha}g_{b-}(x)\, \mathrm{d}x.\label{eq:frac int}
				\end{equation}
				If $a\leq c<d\leq b$ then we have 
				\[
				\int_{c}^{d}f\, dg:=\int_{a}^{b} {1\!\!}1_{(c,d)}f\, \mathrm{d}g.
				\]
				Then for $g\in	W_{[a,b]}^{1-\alpha,\infty}$ and $f\in W_{[a,b]}^{\alpha,1}$, we have the estimate
				\begin{equation}
					\left| \int_{a}^{b}f\, \mathrm{d}g \right| \leq \Lambda_{\alpha}(g) \parallel f \parallel_{\alpha,1,[a,b]},\label{Int-est}
				\end{equation}
				where
				\begin{align*}
				\Lambda_{\alpha}(g):&= \dfrac{1}{\Gamma(1-\alpha)}\underset{a<s<t<b}{\sup}\left|D_{t-}^{1-\alpha}g_{t-}(s)\right| \\
				&\leq \dfrac{1}{\Gamma(1-\alpha)\Gamma(\alpha)}\left\Vert g \right\Vert_{1-\alpha,\infty,[a,b]}<+\infty.
				\end{align*}
				Now we define the integral with respect to fBm as a generalized Lebesgue-Stieltjes integral. It follows from the H\"older continuity of $B^{H}$ that $D_{b-}^{1-\alpha}B_{b-}^{H}\in L^{\infty}[a,b]$ almost surely (a.s.~for short). Then, for a function $f$ with $D_{a+}^{\alpha}f\in L^{1}[a,b]$, we can define the integral with respect to $B^{H}$ through (\ref{eq:frac int}). Since the random variable $\left\Vert B^{H}\right\Vert_{1-\alpha,\infty,[a,b]}$ has moments of all orders, see Lemma 7.5 in Nualart and Rascanu \cite{NR}, the stochastic integral with respect to the fBm satisfies the following estimate
				\begin{equation}\label{FBMint-est}
					\left|\int_{a}^{b}f(s)\, \mathrm{d}B^{H}(s)\right|
					\leq  \Lambda_{\alpha}(B^H) \parallel f \parallel_{\alpha,1,[a,b]}
				\end{equation}
				Throughout this paper we assume that the coefficients $b$, $\sigma_{W}$ and $\sigma_{H}$ satisfy the following hypotheses : 
				\begin{itemize}
					\item[$(Hb)$] $b : [0,T]\times \mathcal{C}_{r} \longrightarrow \mathbb{R}$, is a measurable function has linear growth with respect the second variable and lipschitz with respect the same variable, that is
					\begin{align*}
						\vert b(t,x) \vert &\leq L_{b}^{2}(1+\parallel x \parallel_{\mathcal{C}_{r}} ), \\
						\vert b(t,x)-b(t,y) \vert &\leq L_{b}^{1} \parallel x-y \parallel_{\mathcal{C}_{r}}. 
					\end{align*}			
					\item[$(H\sigma)$] The function $\sigma : [0,T]\times \mathcal{C}_{r} \longrightarrow \mathbb{R}$, is H\"older continuous in $t$ for some $\beta \in (0,1)$, lipschitz continuous with respect the second variable and has linear growth assumption in $x$, that is 
					\begin{align*}
						\vert \sigma(t,x) \vert &\leq L_{\sigma}^{1}(1+\parallel x \parallel_{\mathcal{C}_{r}} ), \\
						\vert \sigma(t,x)-\sigma(s,x)  \vert &\leq L_{\sigma}^{3}\vert t-s \vert^{\beta}\left(1+\parallel x \parallel_{\mathcal{C}_{r}}\right), \\
						\vert\sigma(t,x)-\sigma(t,y) \vert &\leq L_{\sigma}^{2}\parallel x-y \parallel_{\mathcal{C}_{r}}.
					\end{align*}
				\end{itemize}
			\section{Deterministic functional differential equation}
				In this section, we present deterministic results. We start by fixing $\alpha \in (0,\frac{1}{2})$ and $g\in	W_{[0,T]}^{1-\alpha,\infty}$. Moving forward, we establish a uniform partition of the interval $[0, T]$, using a step size denoted as $\delta_{n} = T/2^{n}$, where $\left\lbrace t_{i}^{n} = i\delta_{n}, i = 0,1,...,2^{n} \right\rbrace$. We define $k_{n}(s)$ as $i\delta_{n}$ for $s \in [0, T]$, where $i$ is the largest integer such that $i\delta_{n} \leq s < (i+1)\delta_{n}$. The approximation scheme for equation ~\eqref{equ} is then constructed as follows.
				\begin{equation}
						\left\{ \begin{array}{lll}
							x^{n}(t) & = & \eta(0)+ \dint_{0}^{t}b\left(k_{n}(s),x_{k_{n}(s)}^{n}\right)\,\mathrm{d}s+\dint_{0}^{t}\sigma\left(k_{n}(s),x_{k_{n}(s)}^{n}\right)\,\mathrm{d}g_{s}+\underset{0 \leq s \leq t}{\sup}\left(z^{n}(s)^{-}\right),\,\, t\in[0,T], \\
							x^n(t) & = & \eta(t),\,\,t\in[-r,0].\label{app sch} 
						\end{array}\right.
					\end{equation}			
					With
					\begin{equation}
						\left\{ \begin{array}{lll}
						z^{n}(t) & = & \eta(0)+ \dint_{0}^{t}b\left(k_{n}(s),x_{k_{n}(s)}^{n}\right)\,\mathrm{d}s+\dint_{0}^{t}\sigma\left(k_{n}(s),x_{k_{n}(s)}^{n}\right)\,\mathrm{d}g_{s},\,\, t\in[0,T], \\
						z^n(t) & = & \eta(0),\,\,t\in[-T,0].\label{app sch} 
						\end{array}\right.
					\end{equation}		
					\begin{lemma}
						Let $\theta \in (0,1)$. Then for all $y \in [0,T]$ the following estimate
									\begin{equation}
										\left\Vert x_{y}^{n}-x_{k_{n}(y)}^{n} \right\Vert_{\mathcal{C}_r} \leq C(1+\Lambda_{\alpha}(g))\delta_{n}^{\theta\wedge(1-\alpha)}\left(1+\parallel x^{n} \parallel_{\infty,k_{n}(y)}\right), \label{est incre}
									\end{equation}
						holds.
					\end{lemma}
				\begin{proof}
					Let $y \in [0,T]$, we have $$\left\Vert x_{y}^{n}-x_{k_{n}(y)}^{n} \right\Vert_{\mathcal{C}_r} = \underset{-r \leq v \leq 0}{\sup}\left|x^{n}(y+v)-x^{n}(k_{n}(y)+v)\right|.$$ We partition the above quantity into three segments based on the position of $v$ within the interval $[-r,0]$. Each segment will be analyzed individually. First if $y \leq -v $, given that $\eta$ is $\theta$-H\"older continuous, this implies 
							\begin{align*}
								\left|x^{n}(y+v)-x^{n}(k_{n}(y)+v)\right|& = \left|\eta(y+v)-\eta(k_{n}(y)+v)\right| \\
								&\leq  \mathfrak{H}_{\theta,[-r,0]}(\eta) \delta_{n}^{\theta}.
							\end{align*}
					Second if $k_{n}(y) \leq -v \leq y$, we have
							\begin{align*}
								\left|x^{n}(y+v)-x^{n}(k_{n}(y)+v)\right| &\leq \left|x^{n}(y+v)-x^{n}(0)\right|+\left|x^{n}(0)-x^{n}(k_{n}(y)+v)\right| \\
								&\leq \left|x^{n}(y+v)-x^{n}(0)\right|+ \mathfrak{H}_{\theta,[-r,0]}(\eta)\left|y-k_{n}(y)\right|^{\theta} \\
								& \leq \left|x^{n}(y+v)-x^{n}(0)\right|+ \mathfrak{H}_{\theta,[-r,0]}(\eta) \delta_{n}^{\theta}.
							\end{align*}
					Using the fact that the fonction $x \rightarrow x^-$ is subadditive and by a few simple manipulations we get that 
							\begin{align*}
								\left|x^{n}(y+v)-x^{n}(0)\right| &= \int_{0}^{y+v}b\left(k_{n}(s),x_{k_{n}(s)}^{n}\right)\mathrm{d}s+\int_{0}^{y+v}\sigma\left(k_{n}(s),x_{k_{n}(s)}^{n}\right)\mathrm{d}g_{s}+\underset{0 \leq u \leq y+v}{\sup}\left(z^{n}(u)^{-}\right) \\
								&\leq 2 \left\lbrace \underset{0 \leq u \leq y+v}{\sup}\left|\int_{0}^{u}b\left(k_{n}(s),x_{k_{n}(s)}^{n}\right)\mathrm{d}s\right|
								+ \underset{0 \leq u \leq y+v}{\sup}\left|\int_{0}^{u}\sigma\left(k_{n}(s),x_{k_{n}(s)}^{n}\right)\mathrm{d}g_{s}\right|	\right\rbrace \\
								&=2\left(A_{1}+A_{2}\right).
							\end{align*}
					Let's start with the estimation of the term $A_{1}$. Using the linear growth assumption of $b$ and the change of variables $z=s+k_{n}(y)$,  leads to 
							\begin{align*}
								A_{1} &\leq C\int_{0}^{y+v}\left(1+\parallel x_{k_{n}(s)}^{n} \parallel_{\mathcal{C}_r} \right) \mathrm{d}s \\
								&\leq \int_{k_{n}(y)}^{y+v+k_{n}(y)}\left(1+\left\Vert x^{n} \right\Vert_{\infty,k_{n}(z)} \right) \mathrm{d}z \\
								&\leq \int_{k_{n}(y)}^{y}\left(1+\left\Vert x^{n} \right\Vert_{\infty,k_{n}(z)} \right) \mathrm{d}z
							\end{align*}
					Hence $$A_{1} \leq C\left(y-k_{n}(y)\right)\left(1+ \left\Vert x^{n} \right\Vert_{\infty,k_{n}(y)}\right).$$ Now for $A_{2}$, fisrt set $f(s)=\sigma\left(k_{n}(s),x_{k_{n}(s)}^{n}\right)$. Using the estimate \eqref{Int-est}, the linear growth assumption of $b$ and the change of variables $z=s+k_{n}(y)$, it yields 
							\begin{align*}
								\left|\int_{0}^{u}f(s)\mathrm{d}g_{s}\right|&\leq \Lambda_{\alpha}(g)\int_{0}^{u}\left(\left|f(s)\right|s^{-\alpha}+\int_{0}^{s}\left|f(s)-f(z)\right|(s-z)^{-\alpha-1}\mathrm{d}z\right)\mathrm{d}s \\
								&\leq \int_{k_{n}(y)}^{y}\left[\left(1+\parallel x^{n} \parallel_{\infty,k_{n}(x)}\right)(x-k_{n}(y))^{-\alpha}\right.\\
								&\left.+\int_{0}^{x-k_{n}(y)}\left|f(x-k_{n}(y))-f(z)\right|(x-k_{n}(y)-z)^{-\alpha-1}\mathrm{d}z\right] \mathrm{d}x,
							\end{align*}	
					for any $u \in [0,y+v] $. It's easy to see that	
							$$
								\int_{0}^{x-k_{n}(y)}\left|f(x-k_{n}(y))-f(z)\right|(x-k_{n}(y)-z)^{-\alpha-1}\mathrm{d}z=0,\quad \forall x \in [k_{n}(y),y[,
							$$
					 then  
					 		$$
					 			A_{2} \leq C  \Lambda_{\alpha}(g)(y-k_{n}(y))^{1-\alpha}\left(1+\parallel x^{n} \parallel_{\infty,k_{n}(y)}\right).
					 		$$ 
					 Finaly if $-v \leq k_{n}(y)$, we have 		
					 		$$
					 			\left|x^{n}(y+v)-x^{n}(k_{n}(y)+v)\right|=\left|z^{n}(y+v)-z^{n}(k_{n}(y)+v)+\underset{0 \leq u \leq y+v}{\sup}\left(z^{n}(u)^{-}\right)-\underset{0 \leq u \leq k_{n}(y)+v}{\sup}\left(z^{n}(u)^{-}\right)\right|.
					 		$$ 
					 First let us not that 
					 		$$
					 			\underset{0\leq u \leq k_{n}(y)+v}{\sup}\left(z^{n}(u)^{-}\right)\geq \underset{-v \leq u \leq y}{\sup}\left(z^{n}\left((k_{n}(u)+v)\vee 0\right)^{-}\right).
					 		$$		
					 Fromwhere 
							$$
								\left|x^{n}(y+v)-x^{n}(k_{n}(y)+v)\right| \leq 2\underset{-v \leq u \leq y}{\sup}\left|z^{n}(u+v)-z^{n}\left((k_{n}(u)+v)\vee 0\right)\right|.
							$$ 
					 Let $ -v \leq u \leq y$. First case if $k_{n}(u)+v \leq 0$, following the same steps as $A_{1}$ and $A_{2}$ we obtain that 
					 		$$
					 			\left|z^{n}(u+v)-z^{n}(0)\right| \leq C\left\lbrace \left((u-k_{n}(y))+\Lambda_{\alpha}(g)(u-k_{n}(u))^{1-\alpha}\right) \left(1+\parallel x^{n} \parallel_{\infty,k_{n}(u)}\right)\right\rbrace.
					 		$$ 
					 Second case if $k_{n}(u)+v>0$, we have
								\begin{align*}
									\left|z^{n}(u+v)-z^{n}(k_{n}(u)+v)\right| &\leq \left|\int_{k_{n}(u)+v}^{u+v}b\left(k_{n}(s),x_{k_{n}(s)}^{n}\right)\mathrm{d}s\right|+\left|\int_{k_{n}(u)+v}^{u+v}\sigma\left(k_{n}(s),x_{k_{n}(s)}^{n}\right)\mathrm{d}g_{s}\right| \\
									&=I_{1}+I_{2}.
								\end{align*}
					Using the linear growth assumption in $(Hb)$ and the change of variable to $z=s-v$, we obtain 
								\begin{align*}
									I_{1} &\leq C \int_{k_{n}(u)+v}^{u+v}\left(1+\left\Vert x^{n} \right\Vert_{\infty,k_{n}(s)}\right)\mathrm{d}s \\
									&\leq C\int_{k_{n}(u)}^{u}\left(1+\parallel x^{n} \parallel_{\infty,k_{n}(z)}\right)\mathrm{d}z \\
									&\leq C(u-k_{n}(u))\left(1+\parallel x^{n} \parallel_{\infty,k_{n}(u)}\right).
								\end{align*}				
					For $I_{2}$, using for the second time the estimate \eqref{Int-est}, and changing the integration variable we obtain 
								\begin{align*}
									\left| \int_{k_{n}(u)+v}^{u+v}f(s) \mathrm{d}g_{s} \right| &\leq \Lambda_{\alpha}(g)\left(\int_{k_{n}(u)}^{u}\left(1+\parallel x^{n} \parallel_{\infty,k_{n}(x)} \right)(x-k_{n}(u))^{-\alpha}\mathrm{d}x \right. \\
									&\left.+\int_{k_{n}(u)+v}^{u+v}\int_{k_{n}(u)+v}^{s}\left|f(s)-f(z)\right|(s-z)^{-\alpha-1}\mathrm{d}z\mathrm{d}s \right) \\
									&=\Lambda_{\alpha}(g)(I_{21}+I_{22})
								\end{align*}			
					 Let us assume that $t_{j}^{n} \leq -v <t_{j+1}^{n}$ and $t_{i}^{n} \leq u <t_{i+1}^{n}$ where $j=[\frac{-v}{\delta_{n}}]$ and $i=[\frac{u}{\delta_{n}}]$. We can choose $j+1 \leq i$ for the sheer fact that $-v < k_n(u) \Leftrightarrow k_{n}(-v)+\delta_n \leq u$. Hence $u+v \in ]t_{i-j-1}^{n}, t_{i-j+1}^{n}[ $ and $k_{n}(u)+v \in ]t_{i-j-1}^{n}, t_{i-j}^{n}]$. If $u+v \in ]t_{i-j-1},t_{i-j}^{n}[$, then $I_{22}=0$. From now we suppose $ u+v \in [t_{i-j}, t_{i-j+1}^{n}[$, furthermore
							\begin{align*}
								I_{22}&=\int_{k_{n}(u)+v}^{t_{i-j}^{n}}\int_{k_{n}(u)+v}^{s}\left|f(s)-f(z)\right|(s-z)^{-\alpha-1}\mathrm{d}z\mathrm{d}s+\int_{t_{i-j}^{n}}^{u+v}\int_{k_{n}(u)+v}^{s}\left|f(s)-f(z)\right|(s-z)^{-\alpha-1}\mathrm{d}z\mathrm{d}s \\
								&=\int_{t_{i-j}^{n}}^{u+v}\int_{k_{n}(u)+v}^{s}\left|f(s)-f(z)\right|(s-z)^{-\alpha-1}\mathrm{d}z\mathrm{d}s \\
								&=\int_{t_{i-j}^{n}}^{u+v}\int_{k_{n}(u)+v}^{t_{i-j}^{n}}\left|f(s)-f(z)\right|(s-z)^{-\alpha-1}\mathrm{d}z\mathrm{d}s+\int_{t_{i-j}^{n}}^{u+v}\int_{t_{i-j}^{n}}^{s}\left|f(s)-f(z)\right|(s-z)^{-\alpha-1}\mathrm{d}z\mathrm{d}s \\
								&=\int_{t_{i-j}^{n}}^{u+v}\left|f(t_{i-j}^{n})-f(t_{i-j-1}^{n})\right|\int_{k_{n}(u)+v}^{t_{i-j}^{n}}(s-z)^{-\alpha-1}\mathrm{d}z\mathrm{d}s \\
								&\leq C\left(1+\parallel x^{n} \parallel_{\infty,k_{n}(u)}\right)\int_{t_{i-j}^{n}}^{u+v}(s-t_{i-j}^{n})^{-\alpha}\mathrm{d}s.
							\end{align*} 
					Changing the integration variable to $x=s-t_{i-j}^{n}+k_{n}(u)$, we obtain 
								\begin{align*}
									\int_{t_{i-j}^{n}}^{u+v}(s-t_{i-j}^{n})^{-\alpha}\mathrm{d}s&=\int_{k_{n}(u)}^{u+v-t_{i-j}^{n}+k_{n}(u)}(x-k_{n}(u))^{-\alpha}\mathrm{d}x \\
									&\leq \int_{k_{n}(u)}^{u}(x-k_{n}(u))^{-\alpha}\mathrm{d}x=C(u-k_{n}(u))^{1-\alpha}. \\
								\end{align*}
					Combining all the above estimate leads to $$\underset{-r \leq v \leq 0}{\sup}\left|x^{n}(y+v)-x^{n}(k_{n}(y)+v)\right| \leq C(1+\Lambda_{\alpha}(g))\delta_{n}^{\theta\wedge(1-\alpha)}\left(1+\parallel x^{n} \parallel_{\infty,k_{n}(y)}\right).$$
				\end{proof}
					\begin{proposition}\label{est-frac}
						Let $0 < \alpha < \theta \wedge \frac{1}{2} \wedge \beta $. Then for all $0 \leq s < t \leq T$, we have 
								$$
									 \left| \int_{s}^{t}\sigma\left(k_{n}(u),x_{k_{n}(u)}^{n}\right)\mathrm{d}g_{u} \right| \leq C(1+\Lambda_{\alpha}(g))\int_{s}^{t}(u-s)^{-\alpha}\left(1+\left\Vert x^{n} \right\Vert_{\infty,\alpha,k_{n}(u)}\right)\mathrm{d}u.
								$$
					\end{proposition}	
				\begin{proof}
					 Let $0 \leq s < t \leq T$ and define $h(s)=\sigma(s,x_{s}^n)$. This inequality can be readily derived when both s and t fall within the interval $[t_{i}^{n},t_{i+1}^{n}[$, to the sheer fact that $h(t_{i}^{n})=h(k_{n}(s))$ if s belongs to $[t_{i}^{n},t_{i+1}^{n}[$. From this point onwards, we'll assume that $t_{i}^{n} \leq s < t_{i+1}^{n}\leq t_{j}^{n} \leq t < t_{j+1}^{n}$ with $i=\left[\frac{s}{\delta_{n}} \right]$ and $j=\left[\frac{t}{\delta_{n}} \right]$. Thanks to the linear growth assumption of $\sigma$ and the estimation \eqref{Int-est}, we can conclude
					\begin{align*}
						\left| \int_{s}^{t}h(k_{n}(u))\,dg_{u} \right| 
						&\leq \Lambda_{\alpha}(g)\left[ \int_{s}^{t}\left(1+\left\Vert x^{n}\right\Vert_{\infty,k_{n}(z)} \right)(z-s)^{-\alpha}dz\right. \\ &\left. + \int_{s}^{t}\int_{s}^{z}|h\left( k_{n}(z) \right)-h\left( k_{n}(y)\right) |(z-y)^{-\alpha-1}\, dy\,dz  \right] \\
						&=\left[ I_{1}+I_{2} \right].
					\end{align*}
					Starting with $I_{2}$. Given that $h(k_n(.))$ is piecewise constant, we have the following:
					\begin{align*}
						I_{2}&=\int_{k_{n}(s)+\delta_{n}}^{t}\int_{s}^{k_{n}(z)}|h\left( k_{n}(z) \right)-h\left( k_{n}(y) \right)|(z-y)^{-\alpha-1}\, dy\,dz  \\
						&\leq \int_{k_{n}(s)+\delta}^{t}\int_{s}^{k_{n}(z)}|h\left( k_{n}(z) \right)-h(y)|(z-y)^{-\alpha-1}\, dy\,dz \\
						&+\int_{k_{n}(s)+\delta_{n}}^{t}\int_{s}^{k_{n}(z)}|h(y)-h\left( k_{n}(y) \right)|(z-y)^{-\alpha-1}\, dy\,dz \\		
						& \leq \int_{k_{n}(s)+\delta_{n}}^{t}\int_{s}^{k_{n}(z)}|h\left( k_{n}(z) \right)-h(y)|(k_{n}(z)-y)^{-\alpha-1}\, dy\,dz  \\		
						& +\int_{k_{n}(s)+\delta_{n}}^{t}\int_{s}^{k_{n}(z)}|h(y)-h\left( k_{n}(y) \right)|(z-y)^{-\alpha-1}\, dy\,dz \\
						&= I_{21}+I_{22}.
					\end{align*}
					This can be deduced from the assumption $(H\sigma)$ and the condition $\alpha < \beta$, which implies that 
					\begin{align*}
									I_{21} &\leq C\int_{k_{n}(s)+\delta_{n}}^{t}\left(\int_{s}^{k_{n}(z)}(k_{n}(z)-y)^{\beta-\alpha-1}\left(1+\left\Vert x_{k_{n}(z)}^{n}\right\Vert_{\mathcal{C}_r}\right)\mathrm{d}y \right.\\ &\left.+\int_{s}^{k_{n}(z)}\left\Vert x_{k_{n}(z)}^{n}-x_{y}^{n} \right\Vert_{\mathcal{C}_r}(k_{n}(z)-y)^{-\alpha-1}\mathrm{d}y\right)\mathrm{d}z \\
									&\leq C\int_{k_{n}(s)+\delta_{n}}^{t}\left(\int_{s}^{k_{n}(z)}(k_{n}(z)-y)^{\beta-\alpha-1}\left(1+\left\Vert x_{k_{n}(z)}^{n}\right\Vert_{\mathcal{C}_r}\right)\mathrm{d}y \right.\\
									&\left.+\int_{s}^{k_{n}(z)}\left\Vert x^{n}(k_{n}(z)-y+.)-x^{n}(.) \right\Vert_{\infty,y}(k_{n}(z)-y)^{-\alpha-1}\mathrm{d}y\right)\mathrm{d}z,														
					\end{align*}
					leadin to 
					\begin{align*}
					I_{21}	&\leq C\int_{k_{n}(s)+\delta_{n}}^{t}\left( (k_{n}(z)-s)^{\beta-\alpha}\left(1+\left\Vert x_{k_{n}(z)}^{n}\right\Vert_{\mathcal{C}_r}\right)+\left\Vert x^{n} \right\Vert_{\infty,\alpha,k_{n}(z)}\right) \mathrm{d}z \\
						&\leq C \int_{s}^{t}(z-s)^{-\alpha}\left(1+\left\Vert x^{n} \right\Vert_{\infty,\alpha,k_{n}(z)}\right)\mathrm{d}z.
					\end{align*}
					For $I_{22}$, we can break it down as follows 
						\begin{align*}
								I_{22}
								&= \int_{k_{n}(s)+\delta}^{t}\int_{s}^{k_{n}(s)+\delta_n}|h(y)-h\left( k_{n}(y) \right)|(z-y)^{-\alpha-1}\, dy\,dz \\
								& 
								+\int_{k_{n}(s)+\delta_n}^{t}\int_{k_{n}(s)+\delta_n}^{k_{n}(z)}|h(y)-h\left( k_{n}(y) \right)|(z-y)^{-\alpha-1}\, dy\,dz \\
								& =I_{221}+I_{222}.
							\end{align*}
					On the other hand, utilizing estimate ~\eqref{est incre} and the assumption in $(H\sigma)$, we obtain 
								\begin{align}
									\left|h(y)-h(k_{n}(y))\right| &\leq C\left[\left\Vert x_{y}^{n}-x_{k_{n}(y)}^{n} \right\Vert_{\mathcal{C}_r} +(y-k_{n}(y))^{\beta}\left(1+\left\Vert x_{k_{n}(y)}^{n} \right\Vert_{\mathcal{C}_r}\right)\right] \nonumber \\
									&\leq C \left[(1+\Lambda_{\alpha}(g))\delta_{n}^{\theta\wedge (1-\alpha)}\left(1+\left\Vert x^{n}\right\Vert_{\infty,k_{n}(y)}\right)+\delta_{n}^{\beta}\left(1+\left\Vert x^{n} \right\Vert_{\infty,k_{n}(y)}\right)\right] \nonumber \\
									&\leq C(1+\Lambda_{\alpha}(g))\delta_{n}^{\theta\wedge (1-\alpha)\wedge \beta}\left(1+\left\Vert x^{n} \right\Vert_{\infty,k_{n}(y)}\right). \label{est incr 2}
								\end{align}	
					Therefore, we have 
							$$
								I_{221}\leq C(1+\Lambda_{\alpha}(g))\delta_{n}^{\theta\wedge (1-\alpha)\wedge \beta}\left(1+\parallel x^{n} \parallel_{\infty,k_{n}(s)}\right)\int_{k_{n}(s)+\delta_{n}}^{t}\int_{s}^{k_{n}(s)+\delta_{n}}(z-y)^{-\alpha-1}\mathrm{d}y\mathrm{d}z.
							$$ 
					Applying Fubini's theorem, a change of variables, and some straightforward calculations, we obtain  
								\begin{align*}
									I_{221} &\leq C(1+\Lambda_{\alpha}(g))\delta_{n}^{\theta\wedge (1-\alpha)\wedge \beta}\left(1+\parallel x^{n} \parallel_{\infty,k_{n}(s)}\right)\int_{s}^{k_{n}(s)+\delta_{n}}(k_{n}(s)+\delta_{n}-y)^{-\alpha} \mathrm{d}y \\
									&\leq C(1+\Lambda_{\alpha}(g))\delta_{n}^{\theta\wedge (1-\alpha)\wedge \beta}\left(1+\parallel x^{n} \parallel_{\infty,k_{n}(s)}\right)\int_{s}^{k_{n}(s)+\delta_{n}}(r-s)^{-\alpha} \mathrm{d}r \\
									&\leq C(1+\Lambda_{\alpha}(g))\int_{s}^{t}\left(1+\parallel x^{n} \parallel_{\infty,k_{n}(r)}\right)(r-s)^{-\alpha} \mathrm{d}r.
								\end{align*}		
					Finally, concerning $I_{222}$, it is noteworthy that if $i+1=j$, then $I_{222}=0$. We will now assume that $i+1<j$. Additionally, for any $y \in [t_{k}^{n},t_{k+1}^{n}[$, we have 
						$$
							\left|h(y)-h(k_{n}(y))\right|\leq C(1+\Lambda_{\alpha}(g))\delta_{n}^{\theta\wedge (1-\alpha)\wedge \beta}\left(1+\parallel x^{n} \parallel_{\infty,t_{k}^{n}}\right).
						$$
				 	Finally, by applying Fubini's theorem once again (where $y < k_{n}(z)\Leftrightarrow k_{n}(y)+\delta_{n} \leq z$) and taking into account that $k_{n}(.)$ is piecewise constant, we can derive 
								\begin{align*}
									I_{222} &\leq \int_{k_{n}(s)+\delta_{n}}^{k_{n}(t)}\left|h(y)-h(k_{n}(y))\right|\int_{k_{n}(y)+\delta_{n}}^{t}(z-y)^{-\alpha}\mathrm{d}z \mathrm{d}y \\
									&\leq \int_{k_{n}(s)+\delta_{n}}^{k_{n}(t)}\left|h(y)-h(k_{n}(y))\right|(k_{n}(y)+\delta_{n}-y)^{-\alpha}\mathrm{d}y
								\end{align*}
					By utilizing the inequality ~\eqref{est incr 2}, we obtain 
								\begin{align*}
									I_{222} &\leq C(1+\Lambda_{\alpha}(g))\delta_{n}^{\theta\wedge (1-\alpha)\wedge \beta} \sum_{k=i+1}^{j-1}\int_{t_{k}^{n}}^{t_{k+1}^{n}}\left(1+\parallel x^{n} \parallel_{\infty,t_{k}^{n}}\right)(t_{k+1}^{n}-y)^{-\alpha}\mathrm{d}y \\
									&\leq C(1+\Lambda_{\alpha}(g))\delta_{n}^{\theta\wedge (1-\alpha)\wedge \beta} \sum_{k=i+1}^{j-1}\left(1+\parallel x^{n} \parallel_{\infty,t_{k}^{n}}\right)\int_{t_{k}^{n}}^{t_{k+1}^{n}}(t_{k+1}^{n}-y)^{-\alpha}\mathrm{d}y \\
									&\leq C(1+\Lambda_{\alpha}(g))\delta_{n}^{(\theta-\alpha)\wedge (1-2\alpha)\wedge (\beta-\alpha)} \sum_{k=i+1}^{j-1}\left(1+\parallel x^{n} \parallel_{\infty,t_{k}^{n}}\right) \delta_{n} \\
									&\leq C(1+\Lambda_{\alpha}(g))\delta_{n}^{(\theta-\alpha)\wedge (1-2\alpha)\wedge (\beta-\alpha)}\int_{k_{n}(s)+\delta_{n}}^{k_{n}(t)}\left(1+\parallel x^{n} \parallel_{\infty,k_{n}(u)}\right) \mathrm{d}u.
								\end{align*}	
				Consequently we have 
							$$
								I_{222} \leq C(1+\Lambda_{\alpha}(g))\int_{s}^{t}\left(1+\left\Vert x^{n} \right\vert_{\infty,\alpha,k_{n}(u)}\right) \mathrm{d}u.$$ 
				Combining all the above estimate leads to 
							$$
								 \left| \int_{s}^{t}\sigma\left(k_{n}(u),x_{k_{n}(u)}^{n}\right)\mathrm{d}g_{u} \right| \leq C(1+\Lambda_{\alpha}(g))\int_{s}^{t}(u-s)^{-\alpha}\left(1+\left\Vert x^{n} \right\Vert_{\infty,\alpha,k_{n}(u)}\right) \mathrm{d}u.
							$$
				\end{proof}
				we also need the following technical lemma, which contains an inequality we'll need.
								
			\begin{lemma} \label{tech-est} 
				Let $f$ be a continuous function on $[0,T]$ such that $f(0)\geq0$,
					which we extend to $\left[-T,0\right]$ by
					\[
					f(t)=f(0).
					\]
				Let $g$ be the function given by 
					\[
					l(t)=\underset{s\in[0,t]}{\sup}(f(s))^{-},\,\,t\in[0,T].
					\]
					Then for any $0\leq s<t\leq T$, we have 
					\begin{equation}
						\left|l(t)-l(s)\right|\leq\underset{-T\leq u\leq s}{\sup}\left|f\left(t-s+u\right)-f(u)\right|.\label{est-s>0}
					\end{equation}
					\end{lemma}
				 \begin{proof}
				Since $f^{-}$ is a continuous function on the compact $\left[0,t\right]$, then there exists $u_{0}\in[0,t]$ such that
					\[
					l(t)=f(u_{0})^{-},
					\]
				 and 
					\[
					0\leq l(t)-l(s)=f(u_{0})^{-}-\underset{u\in[0,s]}{\sup}(f(u))^{-}.
					\]
				If $u_{0}\in[0,s]$ then $l(t)=l(s)$, there is nothing to prove. 	
				So we assume that $u_{0}\in(s,t]$. First, we set $v_{0}=u_{0}+s-t\in\left[2s-t,s\right]\subset [-T,s]$. Therefore we have two cases to discuss, depending on the sign of $2s-t$.					
				\noindent $\blacktriangleright$ If $v_{0}\in [-T,0]$ then $f(v_{0})=f(0)$ and
					\[
					\underset{u\in[0,s]}{\sup}(f(u))^{-}\geq(f(0))^{-}=(f(v_{0}))^{-}.
					\]
					Consequently we obtain
					\[
					0\leq l(t)-l(s)\leq f(t-s+v_{0})^{-}-(f(v_{0}))^{-}\leq\left|f(t-s+v_{0})-f(v_{0})\right|.
					\]
				\noindent $\blacktriangleright$ If $0\leq v_{0}\leq s$ then 
					\[
					\underset{u\in[0,s]}{\sup}(f(u))^{-}\geq(f(v_{0}))^{-},
					\]
				and
				\[
				0\leq l(t)-l(s)\leq f(t-s+v_{0})^{-}-(f(v_{0}))^{-}\leq\left|f(t-s+v_{0})-f(v_{0})\right|.
				\]
				\end{proof}
				
				\begin{proposition}\label{Prop 1}
					Suppose that both $b$ and $\sigma$ satisfy the conditions outlined in hypotheses $(Hb)$ and $(H\sigma)$, where $0<\alpha<\frac{1}{2}\wedge \theta \wedge \beta$. Under these conditions, the sequence $x^{n}$ remains uniformly bounded.
				\end{proposition}
				\begin{proof}
					For any integer $N \in \mathbb{N}^{*}$ and $t \in [0,T]$, we have
							\begin{align*}
								\parallel x^{n} \parallel_{\infty,\alpha,t}^{2N}&= \left(  \left\Vert x^{n}(s)\right\Vert_{\infty,t} +\int_{0}^{t}\left\Vert x^{n}(t-s+.)-x^{n}(.)\right\Vert_{\infty,s}h(t,s)\mathrm{d}s \right)^{2N} \\
								&\leq C_{N} \left\lbrace \sup_{-r \leq s \leq t} \left|x^{n}(s)\right|^{2N}+\left(\int_{0}^{t}\sup_{-r \leq u \leq s} \left|x^{n}(t-s+u)-x^{n}(u)\right|h(t,s)\mathrm{d}s\right)^{2N} \right\rbrace \\
								&=C_{N}(I_{1}+I_{2}).
							\end{align*}
					In this context, it's important to note that $C_{N}$ represents a positive constant that solely depends on $N$ and the other parameters of the problem. It might vary from one line to the next. Here, $h(t, s) = (t-s)^{-\alpha-1}$. Now, consider $s \in [0, t]$, and employing the convexity inequality for the second time, we can deduce
							\begin{align*}
								\left|x^{n}(s)\right|^{2N} &\leq C_{N}\left\lbrace \eta(0)^{2N}+\sup_{0 \leq u \leq t}\left|\int_{0}^{u}b\left(k_{n}(v),x_{k_{n}(v)}^{n}\right)\mathrm{d}v\right|^{2N} +\sup_{0 \leq u \leq t}\left|\int_{0}^{u}\sigma\left(k_{n}(v),x_{k_{n}(v)}^{n}\right)\mathrm{d}g_{v} \right|^{2N} \right\rbrace. \\
								&=C_{N}\left\lbrace \eta(0)^{2N}+I_{11}+I_{12} \right\rbrace. 
							\end{align*}	
					Regarding $I_{11}$, through the utilization of Holder's inequality and the linear growth assumption in $(Hb)$, we obtain the following  
							\begin{align*}
								I_{11} &\leq C_{N} \int_{0}^{t}\left|b\left(k_{n}(v),x_{k_{n}(v)}^{n}\right)\right|^{2N}\mathrm{d}v \\
								&\leq C_{N}\int_{0}^{t}\left(1+\left\Vert x_{k_{n}(v)}^{n} \right\Vert_{\mathcal{C}_r}^{2N}\right) \mathrm{d}v \\
								&\leq C_{N} \int_{0}^{t}\left(1+\left\Vert x^{n} \right\Vert_{\infty,k_{n}(v)}^{2N} \right) \mathrm{d}v.
							\end{align*}
					For $I_{12}$, employing the estimate from Proposition \ref{est-frac} and Holder's inequality, we can establish that 
							 $$
							 	\left|\int_{0}^{u}\sigma\left(k_{n}(v),x_{k_{n}(v)}^{n}\right)\mathrm{d}g_{v} \right|^{2N} \leq C_{N}\left(1+\Lambda_{\alpha}(g)\right)^{2N}  \int_{0}^{u}v^{-\alpha}\left(1+\left\Vert x^{n} \right\Vert_{\infty,\alpha,k_{n}(v)}^{2N}\right) \mathrm{d}v,
							 $$ 
					for any $u \in [0, t]$. Consequently,
							$$
								I_{12} \leq C_{N}\left(1+\Lambda_{\alpha}(g)\right)^{2N}\int_{0}^{t}v^{-\alpha}\left(1+\left\Vert x^{n} \right\Vert_{\infty,\alpha,k_{n}(v)}^{2N}\right) \mathrm{d}v.
							$$ 
					Therefore, 
							$$
								I_{1} \leq C_{N}\left\lbrace \left\Vert \eta \right\Vert_{\mathcal{C}_r}^{2N}+(1+\Lambda_{\alpha}(g)^{2N})\int_{0}^{t}v^{-\alpha}\left(1+\left\Vert x^{n} \right\Vert_{\infty,\alpha,k_{n}(v)}^{2N}\right) \mathrm{d}v\right\rbrace.
							$$ 
					Now, let's focus on the term $I_{2}$. Consider $s \in [0, t]$, we partition the aforementioned value into three segments based on the placement of u within the interval $[-r,s]$, and each of these segments will be examined individually. For $ u \in [-r,s-t]$, since $\eta$ satisfies  H\"older condition we have
							\begin{align}
								\left|x^{n}(t-s+u)-x^{n}(u)\right|&=\left|\eta(u+t-s)-\eta(u)\right| \nonumber \\
								&\leq \mathfrak{H}_{\theta,[-r,0]}\left|t-s\right|^{\theta}. \label{case 1}
							\end{align}
					For $ u \in (s-t,0]$, we have 
							\begin{align}
								\left|x^{n}(t-s+u)-x^{n}(u)\right| &\leq \left|x^{n}(t-s+u)-x^{n}(0)\right| +\left|x^{n}(0)-x^{n}(u)\right| \nonumber \\
								&\leq \left|x^{n}(t-s+u)-x^{n}(0)\right|+\mathfrak{H}_{\theta,[-r,0]}\left|t-s\right|^{\theta}. \label{case 2}
							\end{align}
					It is easy to see from ~\eqref{app sch} and simple algebra that
			 				\begin{align*}
			 					\left|x^{n}(t-s+u)-x^{n}(0)\right|&= \left| \int_{0}^{t-s+u}b\left(k_{n}(v),x_{k_{n}(v)}^{n}\right)\mathrm{d}v +\int_{0}^{t-s+u}\sigma\left(k_{n}(v),x_{k_{n}(v)}^{n}\right)\mathrm{d}g_{v} \right. \\
			 					&\left.+\sup_{0 \leq v \leq t-s+u}\left(z^{n}(v)^{-} \right)\right| \\
			 					&\leq 2 \left(\sup_{0 \leq v \leq t-s+u}\left|\int_{0}^{v}b\left(k_{n}(y),x_{k_{n}(y)}^{n}\right)\mathrm{d}y\right|+\sup_{0 \leq v \leq t-s+u}\left|\int_{0}^{v}\sigma\left(k_{n}(y),x_{k_{n}(y)}^{n}\right)\mathrm{d}g_{y}\right|\right).
			 				\end{align*}
			 		For $ v \in [0,t-s+u]$, by making use of the linear growth condition specified in $(Hb)$ and performing a change of the integration variable, we derive 
			 				\begin{align*}
			 					\left|\int_{0}^{v}b\left(k_{n}(y),x_{k_{n}(y)}^{n}\right)\mathrm{d}y\right| &\leq C\int_{0}^{t-s}\left(1+\left\Vert x_{k_{n}(y)}^{n}\right\Vert_{\mathcal{C}_r}\right)\mathrm{d}y \\
			 					&\leq C\int_{s}^{t}\left(1+\left\Vert x^{n} \right\Vert_{\infty,k_{n}(z)}\right)\mathrm{d}z.
			 				\end{align*}
			 		For the integral with respect to $g$ applying the estimate from Proposition \ref{est-frac} and considering that $\left\Vert x^{n} \right\Vert_{\infty,\alpha,k_{n}(y)}$ is monotonically increasing with $y$,  along with a similar change in the integration variable, we attain
			 				\begin{align*}
			 					\left|\int_{0}^{v}\sigma\left(k_{n}(y),x_{k_{n}(y)}^{n}\right)\mathrm{d}g_{y}\right|&\leq C(1+\Lambda_{\alpha}(g))\int_{0}^{v}y^{-\alpha}\left(1+\left\Vert x^{n} \right\Vert_{\infty,\alpha,k_{n}(y)}\right) \mathrm{d}y	\\
			 					&\leq C(1+\Lambda_{\alpha}(g))\int_{0}^{t-s}y^{-\alpha}\left(1+\left\Vert x^{n} \right\Vert_{\infty,\alpha,k_{n}(y)}\right) \mathrm{d}y \\
			 					&\leq C(1+\Lambda_{\alpha}(g))\int_{s}^{t}(z-s)^{-\alpha}\left(1+\left\Vert x^{n} \right\Vert_{\infty,\alpha,k_{n}(z)}\right) \mathrm{d}z
			 				\end{align*}
			 		Proceeding from this point, we arrive at the following inequality  
			 				\begin{align}
			 					\left|x^{n}(t-s+u)-x^{n}(u)\right| &\leq C\left(\int_{s}^{t}\left(1+\parallel x^{n} \parallel_{\infty,k_{n}(z)}\right)\mathrm{d}z\right. \nonumber \\
			 					&\left.+(1+\Lambda_{\alpha}(g))\int_{s}^{t}(z-s)^{-\alpha}\left(1+\left\Vert x^{n} \right\Vert_{\infty,\alpha,k_{n}(z)}\right) \mathrm{d}z +\mathfrak{H}_{\theta,[-r,0]}\left|t-s\right|^{\theta} \right) \nonumber\\
			 					&\leq C\left((1+\Lambda_{\alpha}(g))\int_{s}^{t}(z-s)^{-\alpha}\left(1+\left\Vert x^{n} \right\Vert_{\infty,\alpha,k_{n}(z)}\right) \mathrm{d}z +\mathfrak{H}_{\theta,[-r,0]}\left|t-s\right|^{\theta} \right) \label{case 2-1}
			 				\end{align}
			 		This holds true for any $u \in (s-t,0]$. Finaly for $ u \in (0,s]$, using 
			 		lemma \eqref{tech-est} we obtain 
			 		\begin{align*}
			 			\left|x^{n}(u+t-s)-x^{n}(u)\right|&=\left|z^{n}(u+t-s)-z^{n}(u)+\sup_{0 \leq v \leq t-s+u}z^{n}(v)^{-}-\sup_{0 \leq v \leq u}z^{n}(v)^{-}\right| \\
			 			&\leq \left|z^{n}(u+t-s)-z^{n}(u)\right|+\sup_{-T \leq v \leq u}\left|z^{n}(v+t-s)-z^{n}(v)\right| \\
			 			&\leq 2 \sup_{-T \leq v \leq u}\left|z^{n}(v+t-s)-z^{n}(v)\right| \\
			 			&\leq 2\left(\sup_{s-t \leq v \leq s}\int_{v\vee 0}^{t-s+v}\left|b\left(k_{n}(x),x_{k_{n}(x)}^{n}\right)\right|\mathrm{d}x+\sup_{s-t \leq v \leq s}\left|\int_{v\vee 0}^{t-s+v}\sigma\left(k_{n}(x),x_{k_{n}(x)}^{n}\right)\mathrm{d}g_{x}\right|\right).
			 		\end{align*}
			  	By applying the linear growth assumption within the context of $(Hb)$, employing a change of variables, and taking into account that the norm $\left\Vert x^{n} \right\Vert_{\infty,k_{n}(x)}$ is monotonically increasing with respect to the variable x, we deduce the following result for the case where $0<v \leq u \leq s$ 
			  			$$
			  				\int_{v\vee 0}^{t-s+v}\left|b\left(k_{n}(x),x_{k_{n}(x)}^{n}\right)\right|\mathrm{d}x \leq C\int_{s}^{t}\left(1+\left\Vert x^{n} \right\Vert_{\infty,k_{n}(y)}\right)\mathrm{d}y.
			  			$$
			  	 Utilizing the fractional integral with respect to $g$ along with the estimation provided in proposition \ref{est-frac}, and considering the property that $\left\Vert x^{n} \right\Vert_{\infty,\alpha,k_{n}(x)}$ is non-decreasing in x, we derive the following conclusion for the scenario where $0<v \leq u \leq s$ 
			  	 		$$
			  	 			\left|\int_{v\vee 0}^{t-s+v}\sigma\left(k_{n}(x),x_{k_{n}(x)}^{n}\right)\mathrm{d}g_{x}\right| \leq C\left(1+\Lambda_{\alpha}(g)\right)\int_{s}^{t}(y-s)^{-\alpha}\left(1+\left\Vert x^{n} \right\Vert_{\infty,\alpha,k_{n}(y)}\right)\mathrm{d}y.
			  	 		$$
			  	 As a result, we establish the following outcome for any value of $u$ within the interval $[0, s]$
			  				\begin{align}
			  					\left|x^{n}(t-s+u)-x^{n}(u)\right| &\leq C \left\lbrace \int_{s}^{t}\left(1+\left\Vert x^{n} \right\Vert_{\infty,k_{n}(y)}\right)\mathrm{d}y \right. \nonumber\\
			  					&+\left.\left(1+\Lambda_{\alpha}(g)\right)\int_{s}^{t}(y-s)^{-\alpha}\left(1+\left\Vert x^{n} \right\Vert_{\infty,\alpha,k_{n}(y)}\right)\mathrm{d}y \right\rbrace \nonumber\\
			  					&\leq C \left(1+\Lambda_{\alpha}(g)\right)\int_{s}^{t}(y-s)^{-\alpha}\left(1+\left\Vert x^{n} \right\Vert_{\infty,\alpha,k_{n}(y)}\right)\mathrm{d}y. \label{case 3}
			  				\end{align}
			  	Now that we have all the requisite estimates in place, we can proceed to obtain
			  				\begin{align*}
			  					\left|x^{n}(t-s+u)-x^{n}(u)\right| &\leq C\left\lbrace \mathfrak{H}_{\theta,[-r,0]}(t-s)^{\theta}+\left(1+\Lambda_{\alpha}(g)\right)\int_{s}^{t}(y-s)^{-\alpha}\left(1+\left\Vert x^{n} \right\Vert_{\infty,\alpha,k_{n}(y)}\right)\mathrm{d}y\right\rbrace
			  				\end{align*}
			  	For  any $u \in [-r,s]$. As a result  
			  				\begin{align*}
			  					I_{2} &\leq C_{N} \left\lbrace \left(\int_{0}^{t}(t-s)^{\theta-\alpha-1}\mathrm{d}s\right)^{2N} \right. \\
			  					&\left.+\left(1+\Lambda_{\alpha}(g)\right)^{2N}\left(\int_{0}^{t}\int_{s}^{t}(y-s)^{-\alpha}\left(1+\left\Vert x^{n} \right\Vert_{\infty,\alpha,k_{n}(y)}\right)\mathrm{d}y(t-s)^{-\alpha-1}\mathrm{d}s\right)^{2N} \right\rbrace \\
			  					&\leq C_{N}\left\lbrace 1+\left(1+\Lambda_{\alpha}(g)\right)^{2N}\left(\int_{0}^{t}\left(1+\left\Vert x^{n} \right\Vert_{\infty,\alpha,k_{n}(y)}\right)\int_{0}^{y}(y-s)^{-\alpha}(t-s)^{-\alpha-1}\mathrm{d}s\mathrm{d}y\right)^{2N} \right\rbrace.
			  				\end{align*}
			  	By applying Fubini's theorem with a change of variable $s = z - (t - z)y$, we are able to draw the following conclusion 	
			  				$$
			  				  	\int_{0}^{t}\left(1+\left\Vert x^{n} \right\Vert_{\infty,\alpha,k_{n}(y)}\right)\int_{0}^{y}(y-s)^{-\alpha}(t-s)^{-\alpha-1}\mathrm{d}s\mathrm{d}y \leq B(2\alpha,1-\alpha) \int_{0}^{t}\left(1+\left\Vert x^{n} \right\Vert_{\infty,\alpha,k_{n}(y)}\right)(t-y)^{-2\alpha}\mathrm{d}y
			  				$$
			  	Where $B(.,.)$ is the Euler beta function. Hence, through the application of H"older's inequality, it yields 
			  				$$
			  					I_{2} \leq C_{N}\left\lbrace 1+\left(1+\Lambda_{\alpha}(g)\right)^{2N}\int_{0}^{t}\left(1+\left\Vert x^{n} \right\Vert_{\infty,\alpha,k_{n}(y)}^{2N}\right)(t-y)^{-2\alpha}\mathrm{d}y \right\rbrace.
			  				$$ 
			  	By consolidating all the estimates together, we achieve 
			  				\begin{align*}
			  					\left\Vert x^{n} \right\Vert_{\infty,\alpha,t}^{2N} &\leq C_{N}\left\lbrace 1+\left(1+\Lambda_{\alpha}(g)^{2N}\right)\int_{0}^{t}\left((t-y)^{-2\alpha}+y^{-\alpha}\right)\left(1+\left\Vert x^{n} \right\Vert_{\infty,\alpha,k_{n}(y)}^{2N}\right)\mathrm{d}y \right\rbrace. \\
			  					&\leq C_{N}\left\lbrace 1+\left(1+\Lambda_{\alpha}(g)^{2N}\right)\int_{0}^{t}\left((t-y)^{-2\alpha}+y^{-\alpha}\right)\left(1+\left\Vert x^{n} \right\Vert_{\infty,\alpha,y}^{2N}\right)\mathrm{d}y \right\rbrace.
			  				\end{align*}
			  	Consequently, through the application of a Gronwall-type lemma \cite[Lemma 7.6]{NR}, we can infer that
			  				\begin{equation}
			  				\sup_{n \in \mathbb{N}^{*}} \sup_{ t \in [0,T]} \left\Vert x^{n} \right\Vert_{\infty,\alpha,t} < + \infty. \label{unif bor}
			  				\end{equation} 
			  	This implies that the sequence $x^{n}$ is uniformly bounded within the space $C\left([-r,T],\mathbb{R}\right)$.
			  	\end{proof}
			  			\begin{proposition}\label{prop 3.3}
			  				For every positive integer $N$ and $n$, there exists a constant $C_{N}$ such that for all $s$ and $t$ in the interval $[-r, T]$, the following inequality holds 
			  					\begin{equation}\label{equi est}
			  						\left|x^{n}(t)-x^{n}(s)\right| \leq C_{N}\left|t-s\right|^{\theta \wedge \frac{1}{2}},
			  					\end{equation}
			  			\end{proposition}
			  		\begin{proof}
			  		Let $ N,n \in \mathbb{N}^{*}$. The proof of this proposition can be broken down into three distinct cases that need to be examined.\\
			  		\textbf{Step 1:} If $-r \leq s < t \leq 0 $, we have 
			  				$$
			  				\left|x^{n}(t)-x^{n}(s)\right|^{2N} \leq \mathfrak{H}_{\theta,[-r,0]}^{2N}\left|t-s\right|^{2N\theta}.
			  				$$ 
			  		\textbf{Step 2:} In the second case, where $0 < s < t \leq T$, we can start by setting $u = s$ in equation \eqref{case 3}. Using H\"older's inequality and some fundamental estimates, we arrive at the following result
			  				\begin{align*}
			  					\left|x^{n}(t)-x^{n}(s)\right| &\leq C \left(1+\Lambda_{\alpha}(g)\right)\int_{s}^{t}(y-s)^{-\alpha}\left(1+\left\Vert x^{n} \right\Vert_{\infty,\alpha,y}\right)\mathrm{d}y \\
			  					 &\leq  C \left(1+\Lambda_{\alpha}(g)\right)(t-s)^{(1-\alpha)(1-\frac{1}{2N})}\left(\int_{s}^{t}(y-s)^{-\alpha}\left(1+\left\Vert x^{n} \right\Vert_{\infty,\alpha,y}^{2N}\right)\mathrm{d}y\right)^{\frac{1}{2N}}.
			  				\end{align*}
			  		With the help of \eqref{unif bor}, we infer that 
			  				\begin{align*}
			  					\left|x^{n}(t)-x^{n}(s)\right|^{2N} &\leq C_{N} \left(1+\Lambda_{\alpha}(g)\right)^{2N}\left|t-s\right|^{2N(1-\alpha)} \\
			  					&\leq C_{N} \left|t-s\right|^{N}.
			  				\end{align*}
					\textbf{Step 3:} If $ s <0<t$, following a similar procedure as employed for the estimations \eqref{case 2}-\eqref{case 2-1} and utilizing equations \eqref{unif bor}, we arrive at the following conclusions
			  					\begin{align*}
			  						\left|x^{n}(t)-x^{n}(s)\right|^{2N} &\leq  C_{2N}\left((1+\Lambda_{\alpha}(g))^{2N}t^{(1-\alpha)(2N-1)}\int_{0}^{t}z^{-\alpha}\left(1+\left\Vert x^{n} \right\Vert_{\infty,\alpha,z}^{2N}\right) \mathrm{d}z +\mathfrak{H}_{\theta,[-r,0]}^{2N}\left|t-s\right|^{2N\theta } \right) \\
			  						&\leq C_{N}\left|t-s\right|^{2N\theta \wedge N}.
			  					\end{align*}
			  		\end{proof}
			  	In conclusion, the analysis for both cases, i.e., when $s < 0 < t, 0 < s < t \leq T$  and $s < t \leq 0$ has allowed us to derive the desired results, demonstrating the validity of the proposition.
			  	\section{Existence and uniqueness of solutions}
			  	 \subsection{Existence of a solution}
			  	\begin{theorem}\label{Th exi}
			  		Under the conditions that the coefficients $b$ and $\sigma$ adhere to the assumptions $(Hb)$ and $(H\sigma)$, if $0 < \alpha < \frac{1}{2} \wedge \theta \wedge \beta$, then the equation \eqref{equ} possesses a solution denoted as $x \in C^{\theta\wedge\frac{1}{2}}\left([-r, T]\right)$.
			  	\end{theorem}
			  	\begin{proof}
			  	The sequence of functions $x^{n}$ exhibits equicontinuity and is bounded in $C^{\theta \wedge \frac{1}{2}}\left([-r, T], \mathbb{R}\right)$. Consequently, by virtue of the Ascoli theorem, there exists a subsequence $(x^{\phi(n)})_{n \geq 0}$ that uniformly converges to a function $x \in C\left([-r, T], \mathbb{R}\right)$. By leveraging the uniform convergence of $(x^{\phi(n)})_{n \geq 0}$ to $x$, along with the observation that in \eqref{equi est}, the constant $C_N$ remains independent of $n$, it becomes evident that $x$ belongs to $C^{\theta \wedge \frac{1}{2}}$. Furthermore the integral $\int_{0}^{.}\sigma\left(s,x_{s}\right)\mathrm{d}g_{s}$ is well-defined. For the sake of simplicity, we assume that $x^{n}$ converge uniformly to $x$, and let's demonstrate the following convergences
			  				\begin{align}
			  					\underset{n \rightarrow +\infty}{\lim} \int_{0}^{t}b\left(k_{n}(s),x_{k_{n}(s)}^{n}\right)\mathrm{d}s &= \int_{0}^{t}b\left(s,x_{s}\right) \mathrm{d}s, \label{lim 1} \\
			  					\underset{n \rightarrow +\infty}{\lim} \int_{0}^{t}\sigma\left(k_{n}(s),x_{k_{n}(s)}^{n}\right)\mathrm{d}g_{s} &= \int_{0}^{t}\sigma\left(s,x_{s}\right) \mathrm{d}g_{s}. \label{lim 2} 
			  				\end{align}	
			  		These limits hold as $n$ tends to infinity, with uniform convergence for $t$ in the interval $[0, T]$. For the first convergence, by leveraging the assumptions outlined in $(Hb)$ and considering the uniform convergence of $x^{n}$ to $x$, we establish equation \eqref{lim 1}. Regarding the second convergence, we can employ the estimate \eqref{Int-est}, to obtain
			  				\begin{align*}
			  					\underset{0 \leq t \leq T}{\sup}\left|\int_{0}^{t}\sigma\left(s,x_s\right)-\sigma\left(k_n(s),x_{k_{n}(s)}^{n}\right)\mathrm{d}g_{s}\right|&\leq C\Lambda_\alpha(g)\sup_{0 \leq t \leq T}\left|\sigma\left(t,x_t\right)-\sigma\left(k_n(t),x_{k_{n}(t)}^{n}\right)\right|^{\epsilon}\\
			  					&\times (I_{1}+I_{2}+I_{3}).
			  				\end{align*}	
			  		Where 	
			  				\begin{align*}
			  					I_{1}&=\sup_{0 \leq t \leq T}\left|\sigma\left(t,x_t\right)-\sigma\left(k_n(t),x_{k_{n}(t)}^{n}\right)\right|^{1-\epsilon}, \\
			  					I_{2}&=\int_{0}^{T}\int_{0}^{s}\dfrac{\left|\sigma\left(s,x_s\right)-\sigma\left(z,x_z\right)\right|^{1-\epsilon}}{(s-z)^{\alpha+1}}\mathrm{d}z\mathrm{d}s, \\
			  					I_{3}&=\int_{0}^{T}\int_{0}^{s}\dfrac{\left|\sigma\left(k_n(s),x_{k_n(s)}^n\right)-\sigma\left(k_n(z),x_{k_n(z)}^n\right)\right|^{1-\epsilon}}{(s-z)^{\alpha+1}}\mathrm{d}z\mathrm{d}s,
			  				\end{align*}		
			  	where $\epsilon$ is a small positive number, ensuring that $I_{2}$ and $I_{3}$ are finite. Next, we need to verify the convergence of $I_{1}$ by utilizing the estimate presented in \eqref{est incre}, coupled with the equation ~\eqref{unif bor} and the assumption laid out in $(H\sigma)$ we obtain 
			  				\begin{align*}
			  					I_{1} &\leq C\left[\underset{0 \leq t \leq T}{\sup}\left(t-k_{n}(t)\right)^{\beta}\left(1+\left\Vert x_{t}  \right\Vert_{\mathcal{C}_{r}}\right)+\underset{0 \leq t \leq T}{\sup}\left\Vert x_{t}-x_{k_{n}(t)}^{n} \right\Vert_{\mathcal{C}_{r}}\right]^{1-\epsilon} \\
			  					&\leq C \left[\delta_{n}^{\beta}+\underset{0 \leq t \leq T}{\sup}\left\Vert x_{t}-x_{t}^{n} \right\Vert_{\mathcal{C}_{r}}+\underset{0 \leq t \leq T}{\sup}\left\Vert x_{t}^{n}-x_{k_{n}(t)}^{n}\right\Vert_{\mathcal{C}_{r}} \right]^{1-\epsilon} \\
			  					&\leq C\left[\delta_{n}^{\beta}+\underset{0 \leq t \leq T}{\sup}\left|x(t)-x^{n}(t)\right|+\left(1+\Lambda_{\alpha}(g)\right)\delta_{n}^{\theta \wedge (1-\alpha)}\left(1+\left\Vert x^{n} \right\Vert_{\infty,T}\right)\right]^{1-\epsilon} \longrightarrow 0 .
			  				\end{align*}
			   Now, turning our attention to $I_2$, we have 
			  				\begin{align*}
			  					I_{2}&=\int_{0}^{T}\int_{0}^{s}\dfrac{\left|\sigma\left(s,x_s\right)-\sigma\left(z,x_z\right)\right|^{1-\epsilon}}{(s-z)^{\alpha+1}}\mathrm{d}z\mathrm{d}s \\
			  					&\leq C\left[\int_{0}^{T}\int_{0}^{s}\dfrac{(s-z)^{(1-\epsilon)\beta}}{(s-z)^{\alpha+1}}\mathrm{d}z\mathrm{d}s+\int_{0}^{T}\int_{0}^{s}\dfrac{\left\Vert x_{s}-x_{z}\right\Vert_{\mathcal{C}_r}^{1-\epsilon}}{(s-z)^{\alpha+1}}\mathrm{d}z\mathrm{d}s\right] \\
			  					&\leq C_{1}+C_{2}\left(1+\Lambda_{\alpha}(g)\right)^{1-\epsilon}\int_{0}^{T}\int_{0}^{s}\dfrac{(s-z)^{(1-\epsilon)(\theta \wedge \frac{1}{2})}}{(s-z)^{\alpha+1}}\mathrm{d}z\mathrm{d}s <+\infty.
			  				\end{align*}
			  	That has been tied to the fact that $x \in C^{\theta \wedge \frac{1}{2}}$. Finally, considering $I_3$, we find 
			  				\begin{align*}
			  					I_{3}&=\int_{0}^{T}\int_{0}^{s}\dfrac{\left|\sigma\left(k_n(s),x_{k_n(s)}^n\right)-\sigma\left(k_n(z),x_{k_n(z)}^n\right)\right|^{1-\epsilon}}{(s-z)^{\alpha+1}}\mathrm{d}z\mathrm{d}s \\
			  					&\leq C\int_{0}^{T}\int_{0}^{s}\dfrac{\left|k_n(s)-k_n(z)\right|^{(1-\epsilon)\beta}+\left\Vert x_{k_{n}(s)}^{n}-x_{k_{n}(z)}^{n}\right\Vert_{\mathcal{C}_r}^{1-\epsilon}}{(s-z)^{\alpha+1}}\mathrm{d}z\mathrm{d}s.
			  				\end{align*}
			  	On the other hand we have 
			  				$$
			  					\left\Vert x_{k_{n}(s)}^{n}-x_{k_{n}(z)}^{n}\right\Vert_{\mathcal{C}_r} \leq C \left(1+\Lambda_{\alpha}(g)\right)\left|k_{n}(s)-k_{n}(z)\right|^{\theta\wedge \frac{1}{2}}.
			  				$$ 
			  	Hence
			  				 $$
			  				 	I_{3} \leq C\int_{0}^{T}\int_{0}^{s}\dfrac{\left|k_n(s)-k_n(z)\right|^{(1-\epsilon)\beta}+\left|k_{n}(s)-k_{n}(z)\right|^{(1-\epsilon)(\theta \wedge \frac{1}{2})}}{(s-z)^{\alpha+1}}\mathrm{d}z\mathrm{d}s <+\infty.
			  				 $$
			  	Where we have using the partition on the interval and decomposing the integrals in finite sums.			
			  	\end{proof}
			  	 \subsection{Uniqueness of the solution}
			  	  	To establish the uniqueness of the results, it is essential for us to have the capability to manage the disparity between two solutions to our equation $\left\Vert x^1 -x^2\right\Vert_{\infty,\alpha,T}$. However, as we have observed in \cite[Remark 3.6]{FR}, it is not feasible to regulate the disparity between two regulator terms $\underset{0 \leq s \leq .}{\sup}(z^{1}(s))$ and $\underset{0 \leq s \leq .}{\sup}(z^{2}(s))$ using a $\lambda$-H\"older norm. That's why we are constrained to the class of equations in which $\sigma$ is expressed in the following form $\sigma(s,x_s)=\sigma(s,x(s-r))$.
			  	  	\begin{theorem}
			  	  			Assume that, the coefficients $b$ and $\sigma$ satisfy to the assumptions $(Hb)$ and $(H\sigma)$, if $0 < \alpha < \frac{1}{2} \wedge \theta \wedge \beta$, then the equation \eqref{equ} has a unique solution.
			  	  	\end{theorem}
			  	  	\begin{proof}
			  	  		We prove, via inductive argument, that the pathwise uniqueness for equation \eqref{equ} holds. Indeed let $x^i,\, i=1,2$ be two solutions of \eqref{equ} defined. First let us set 
			  	  					\begin{equation}
			  	  						z^i(t)= \eta(0)+\int_{0}^{t}b\left(s,x^i_{s} \right) \mathrm{d}s
			  	  						+ \int_{0}^{t}\sigma\left(s,x^i(s-r) \right) \mathrm{d}g_{s},\,t \in [0,T].
			  	  						\label{Appr eq3}
			  	  					\end{equation} 
			  	  					Since  $x^i,\, i=1,2$ have the same  initial condition $\eta$ on $[-r,0]$ it is clear that for any $ t  \in [0,r]$ we have
			  	  							$$
			  	  								\int_{0}^{t}\sigma(s,x^1(s-r))\mathrm{d}g_{s}=\int_{0}^{t}\sigma(s,x^2(s-r))\mathrm{d}g_{s}=\int_{0}^{t}\sigma(s,\eta (s-r))\mathrm{d}g_{s}.
			  	  							$$
			  	  			Let $t \in [0,r]$, it is easy to see that $$ \underset{0 \leq u \leq t}{\sup}\left|x^1(u)-x^2(u)\right| \leq 2\underset{0 \leq u \leq t}{\sup}\left|z^1(u)-z^2(u)\right|.$$
			  	  			Using the assumption $(Hb)$ and the above inequality it yields 
			  	  			\begin{align*}
			  	  				\underset{0 \leq u \leq t}{\sup}\left|x^1(u)-x^2(u)\right| &\leq 2\underset{0 \leq u \leq t}{\sup} \left|\int_{0}^{u}b(s,x_{s}^1)-b(s,x_{s}^2)\mathrm{d}s\right| \\
			  	  				&\leq 2L_2\underset{0 \leq u \leq t}{\sup} \int_{0}^{u}\left\Vert x_{s}^1-x_{s}^2 \right\Vert_{\mathcal{C}_r} \mathrm{d}s \\
			  	  				&\leq  2L_2\underset{0 \leq u \leq t}{\sup} \int_{0}^{u}\underset{0 \leq v \leq s}{\sup}\left|x^1(v)-x^2(v)\right| \mathrm{d}s \\
			  	  				&\leq 2L_2 \int_{0}^{t}\underset{0 \leq v \leq s}{\sup}\left|x^1(v)-x^2(v)\right| \mathrm{d}s.
			  	  			\end{align*}
			  	  			Thanks to Gronwall's inequality, we have that for all $t \in [0,r]$ $$\underset{0 \leq u \leq t}{\sup}\left|x^1(u)-x^2(u)\right|=0.$$ Using induction, we confirm the pathwise uniqueness within $[0, nr]$ for every $n$, thereby establishing pathwise uniqueness within $[0, T]$.
			  	  	\end{proof}			
			\section{Convergence of the sequence of Euler approximation}
			In this section, we make the assumption that $\sigma(s, x_s) = \sigma(s, x(s - r))$. Our next objective is to demonstrate that the sequence $x^n$, as defined in Equation~\eqref{app sch}, is a Cauchy sequence. From this point forward, we will assume that $b$ is a $\beta$-Holder function with respect to the first variable.
				\begin{proposition}
					 The sequence $x^n$ forms a Cauchy sequence within the space $C\left([-r, T], \mathbb{R}\right)$.
				\end{proposition}
				\begin{proof}
					Let $n$ and $m$ be arbitrary integers such that $n \leq m$, and define $d(t) = x^n(t) - x^m(t)$ for $t \in [-r, T]$. Using Equation~\eqref{app sch}, we can derive the following inequalities for any $t \in [0, T]$ 
						\begin{align*}
							\sup_{0 \leq s \leq t}\left|d(s)\right|\leq & \sup_{0 \leq s \leq t}\left|z^n(s)-z^m(s)\right|+\sup_{0 \leq s \leq t}\left|\sup_{0 \leq u \leq s}(z^n(u)^-)-\sup_{0 \leq u \leq s}(z^m(u)^-)\right| \\
							&\leq 2\sup_{0 \leq s \leq t}\left|z^n(s)-z^m(s)\right|.
						\end{align*}
				let $t \in [0, T]$. Using Equation~\eqref{app sch}, we can express $z^n(t)$ and $z^m(t)$ as follows 
					\begin{align}
						z^n(t)&=z^n(k_n(t))+(t-k_n(t))b\left(k_n(t),x_{k_n(t)}^n\right) \nonumber\\
						&+(g_t-g_{k_n(t)})\sigma\left(k_n(t),x^n(k_n(t)-r)\right), \label{equa1} \\
						z^m(t)&=z^m(k_m(t))+(t-k_m(t))b\left(k_m(t),x_{k_m(t)}^m\right) \nonumber\\
								&+(g_t-g_{k_m(t)})\sigma\left(k_m(t),x^m(k_m(t)-r)\right). \label{equa2}
					\end{align}
				By subtracting Equation~\eqref{equa1} from Equation~\eqref{equa2}, we obtain 
					\begin{align*}
						z^n(t)-z^m(t)&=z^n(k_m(t))-z^m(k_m(t))+(t-k_m(t))\left[b(k_n(t),x_{k_n(t)}^n)-b(k_m(t),x_{k_m(t)}^n)\right] \\
						&+(g_t-g_{k_m(t)})\left[\sigma(k_n(t),x^n(k_n(t)-r))-\sigma(k_m(t),x^m(k_m(t)-r))\right] \\
						&+(k_m(t)-k_n(t))b(k_n(t),x_{k_n(t)}^n)+(g_{k_m(t)}-g_{k_n(t)})\sigma(k_n(t),x^n(k_n(t)-r)) \\
						&+z^n(k_n(t))-z^n(k_m(t)).
					\end{align*}
				Since $n \leq m$, we have $k_m(t) \in [k_n(t), t]$. Therefore, by using Equation~\eqref{equa1}, we can express the difference as follows
					\begin{align*}
						z^n(k_m(t))-z^n(k_n(t))&=(k_m(t)-k_n(t))b(k_n(t),x_{k_n(t)}^n)\\
						&+(g_{k_m(t)}-g_{k_n(t)})\sigma(k_n(t),x^n(k_n(t)-r)).
					\end{align*}
				Hence, we can express the difference between $z^n(t)$ and $z^m(t)$ as	
					\begin{align*}
						z^n(t)-z^m(t)&=z^n(k_m(t))-z^m(k_m(t))+(t-k_m(t))\left[b(k_n(t),x_{k_n(t)}^n)-b(k_m(t),x_{k_m(t)}^n)\right] \\
						&+(g_t-g_{k_m(t)})\left[\sigma(k_n(t),x^n(k_n(t)-r))-\sigma(k_m(t),x^m(k_m(t)-r))\right].
					\end{align*}
				By utilizing the assumption in $(Hb)$, and recognizing that $b$ is a $\beta$-Holder function, we can derive 
					\begin{align*}
						\left|b\left(k_n(u),x_{k_n(u)}^n\right)-b\left(k_m(u),x_{k_m(u)}^m\right)\right| &\leq \left|b\left(k_n(u),x_{k_n(u)}^n\right)-b\left(k_n(u),x_{k_m(u)}^m\right)\right|\\
						&+\left|b\left(k_n(u),x_{k_m(u)}^m\right)-b\left(k_m(u),x_{k_m(u)}^m\right)\right| \\
						&\leq C \left[\left\Vert x_{k_n(u)}^n-x_{k_m(u)}^m \right\Vert_{\mathcal{C}_r}+(k_m(u)-k_n(u))^{\beta}\right].
					\end{align*}
				 On the other hand, according to Proposition \eqref{prop 3.3}, we have
				 				\begin{align*}
				 					\left\Vert x_{k_n(t)}^n-x_{k_m(t)}^m \right\Vert_{\mathcal{C}_r}&=\sup_{-r \leq v \leq 0}\left|x^n(k_n(t)+v)-x^m(k_m(t)+v)\right| \\
				 					&\leq \sup_{-r \leq v \leq 0}\left|x^n(k_m(t)+v)-x^m(k_m(t)+v)\right| \\
				 					&+\sup_{-r \leq v \leq 0}\left|x^n(k_n(t)+v)-x^n(k_m(t)+v)\right| \\
				 					&\leq \sup_{0 \leq s \leq k_m(t)}\left|d(s)\right|+C(k_m(t)-k_n(t))^{\theta\wedge \frac{1}{2}}.
				 				\end{align*}
				Therefore,
				 		\begin{equation}
				 			\left|b\left(k_n(t),x_{k_n(t)}^n\right)-b\left(k_m(t),x_{k_m(t)}^m\right)\right| \leq C\left[\sup_{0 \leq s \leq k_m(t)}\left|d(s)\right|+\delta_{n}^{\theta \wedge \frac{1}{2}\wedge \beta}\right]. \label{est b}
				 		\end{equation}
				Where C is a constant do not depend on $n$ and $m$. Now, let's consider the case where $t \in [0, k_m(r)+\delta_m[$. By using the assumption in $(H\sigma)$, we obtain
				 				\begin{align*}
				 					\left|\sigma\left(k_n(t),x^n(k_n(t)-r)\right)-\sigma\left(k_m(t),x^m(k_m(t)-r)\right)\right|&=\left|\sigma\left(k_n(t),\eta(k_n(t)-r)\right)-\sigma\left(k_m(t),\eta(k_m(t)-r)\right)\right| \\
				 					&\leq C\left[(k_m(t)-k_n(t))^{\beta}+(k_m(t)-k_n(t))^{\theta}\right] \\
				 					&\leq C\delta_{n}^{\beta \wedge \theta}.
				 				\end{align*}
				So that 
					\begin{align*}
						\left|z^n(t)-z^m(t)\right|&\leq \sup_{0 \leq s \leq k_m(t)}\left|z^n(s)-z^m(s)\right|+C\left(\delta_{m}\sup_{0 \leq s \leq k_m(t)}\left|z^n(s)-z^m(s)\right|\right. \\
						&\left.+\delta_{m}\delta_{n}^{\beta \wedge \theta \wedge \frac{1}{2}}+\Lambda_{\alpha}(g)\delta_{m}^{1-\alpha}\delta_{n}^{\theta \wedge \beta}\right) \\
						&\leq (1+C\delta_{m})\sup_{0 \leq s \leq k_m(t)}\left|z^n(s)-z^m(s)\right|+C\left(\delta_{m}\delta_{n}^{\beta \wedge \theta \wedge \frac{1}{2}}\right. \\
						&\left.+\Lambda_{\alpha}(g)\delta_{m}^{1-\alpha}\delta_{n}^{\theta \wedge \beta}\right).
					\end{align*}
				Assuming that $m = n + 1$, we can establish the following equality: $\delta_{n} = 2\delta_{m}$. Consequently, we can derive the following expression
						$$	
							\left|z^n(t)-z^{n+1}(t)\right| \leq (1+C\delta_{n+1})\sup_{0 \leq v \leq k_{n+1}(t)}\left|z^n(v)-z^{n+1}(v)\right|+\delta_{n+1}h(\delta_{n}).
						$$ 
				Here, the function $h(\delta_{n})$ is defined as 
						$$
							h(\delta_{n})=C\left(\delta_{n}^{\beta \wedge \theta \wedge \frac{1}{2}}+\Lambda_{\alpha}(g)\delta_{n}^{\theta-\alpha \wedge \beta-\alpha}\right).
						$$
				Hence 
						$$
							\left|z^n(t)-z^{n+1}(t)\right| \leq (1+C\delta_{n+1})\sup_{0 \leq v \leq t_{i}^{n+1}}\left|z^n(v)-z^{n+1}(v)\right|+\delta_{n+1}h(\delta_{n}), \quad t \in [t_{i}^{n+1},t_{i+1}^{n+1}[.
						$$
				By applying \cite[Lemma 3.2]{CT}, we can conclude 
						$$
							\sup_{0 \leq v \leq r}\left|x^n(v)-x^{n+1}(v)\right| \leq \dfrac{e^{TC}-1}{C}h(\delta_{n}).
						$$ 
				Summing these inequalities, we obtain
						$$
							\sum_{n=0}^{+\infty}\sup_{0 \leq v \leq r}\left|x^n(v)-x^{n+1}(v)\right| \leq \dfrac{e^{TC}-1}{C}\sum_{n=0}^{\infty}h(\delta_{n})<+\infty.
						$$ 
				The uniform convergence of the sequence $(x^n)_{n \geq 0}$ in the interval $[0, r]$ is a direct consequence of the condition $\alpha < \theta \wedge \beta$. This condition ensures the convergence of the sequence within this specific interval. Now, let's consider the case where $t \in [0, 2r]$, $n \leq m$, and $s \in [0, t]$. Making use of \eqref{est b} we obtain
					\begin{align*}
						\left|z^n(s)-z^m(s)\right|&=\left|\int_{0}^{s}\left[b(k_n(u),x_{k_n(u)}^{n})-b(k_m(u),x_{k_m(u)}^{m})\right]\mathrm{d}u\right. \\
						&\left.+\int_{0}^{s}\left[\sigma(k_n(u),x^{n}(k_n(u)-r))-\sigma(k_m(u),x^{m}(k_m(u)-r))\right]\mathrm{d}g_u\right| \\
						&\leq C\left(\delta_{n}^{\theta \wedge \frac{1}{2}\wedge \beta}+\int_{0}^{s}\sup_{0 \leq s \leq k_m(u)}\left|d(s)\right|\mathrm{d}u\right)\\
						&+\sup_{0 \leq v \leq 2r} \left|\int_{0}^{v}\sigma(k_n(u),x^{n}(k_n(u)-r))-\sigma(k_m(u),x^{m}(k_m(u)-r))\mathrm{d}g_u\right|.
					\end{align*}
				Utilizing the result from Equation \eqref{lim 2} and the uniform convergence of $x^n$ on $[0, r]$, we can conclude that for every $\epsilon > 0$, there exists an integer $N_0$ such that for every $N_0 \leq n \leq m$, the following inequality holds 
						$$
							\sup_{0 \leq v \leq 2r} \left|\int_{0}^{v}\sigma(k_n(u),x^{n}(k_n(u)-r))-\sigma(k_m(u),x^{m}(k_m(u)-r))\mathrm{d}g_u\right|<\epsilon.
						$$ 
				Consequently, we can derive the following inequality
						$$
							\sup_{0 \leq s \leq t}\left|x^n(s)-x^m(s)\right| \leq C\left(\delta_{n}^{\theta \wedge \frac{1}{2}\wedge \beta}+\int_{0}^{t}\sup_{0 \leq v \leq u}\left|x^n(v)-x^m(v)\right|\mathrm{d}u\right)+\epsilon.
						$$ 
				Applying Gronwall's lemma, we obtain 
						$$
							\sup_{0 \leq s \leq t}\left|x^n(s)-x^m(s)\right|\leq C (\delta_{n}^{\beta \wedge \theta \wedge \frac{1}{2}}+\epsilon).
						$$
				This demonstrates that $x^n$ converges uniformly in the interval $[0, 2r]$. By repeating these steps, we can establish through an induction argument that $x^n$ is a cauchy sequence. Consequently, our sequence converges, enabling the possibility of obtaining numerical results for solutions.
				\end{proof}
				\begin{examples}
						  		The following equations satisfy our hypothesis 
						  		\begin{itemize}
						  			\item (linear example) for any $a,b \in \mathbb{R}$
						  				\begin{align*}
						  					x(t)&=r+\int_{0}^{t}x(s-r)\mathrm{d}s+\int_{0}^{t}\left(ax(s-r)+b\right)\mathrm{d}g_s+y(t),\quad t \in [0,T], \\
						  					x(t)&=t+r,\quad t \in [-r,0].
						  				\end{align*}
						  			\item (non-linear example) 
						  				\begin{align*}
						  				  	x(t)&=\int_{0}^{t}\cos(x(s))\mathrm{d}s+\int_{0}^{t}\sin(x(s-r)+s)\mathrm{d}g_s+y(t),\quad t \in [0,T], \\
						  				  	x(t)&=t^2,\quad t \in [-r,0].
						  				  \end{align*}
						  		\end{itemize}
					\end{examples}
				\section{Stochastic integrals and equations with respect to the fractional
				Brownian motion} 
				In this section, we apply the deterministic results in order to prove the main theorem of this paper. Let $B^H = {B_{t}^H, t \in [0, T]}$ be a fractional Brownian motion with Hurst parameter $H \in (1/2, 1)$, defined on a complete probability space $(\Omega, \mathcal{F}, \mathbb{P})$. It is a centered Gaussian self-similar process with stationary increments, and satisfies the following property: for every $p > 0$,
					$$
						\mathbb{E}\left[\left|B_{t}^H-B_{s}^H\right|^p\right]=\mathbb{E}\left[\left|B_{1}^{H}\right|^{p}\right]\left|t-s\right|^{pH},
					$$ 
				which implies that there exists a version of the fBm $B^H$ which is a continuous process for every $\gamma < H$, thanks to Kolmogorov's continuity theorem. As a consequence, if $u = {u_t, t \in [0, T]}$ is a stochastic process whose trajectories belong to the space $W_{[0, T]}^{\alpha, 1}$, with $1 - H < \alpha < 1/2$, the pathwise integral $\int_{0}^{.} u_s , \mathrm{d}B_{s}^H$ exists and we have the following estimate:
					$$
						\left|\int_{0}^{T}u_s\mathrm{d}B_{s}^H\right| \leq \Lambda_{\alpha}(B^H) \parallel u \parallel_{\alpha,1,[0,T]},
					$$
				where $\Lambda_{\alpha}(B^H)$ is a random variable that has moments of all orders (see \cite[Lemma 7.5]{NR}). Moreover, if the trajectories of the process $u$ belong to the space $W_{[0, T]}^{\alpha, \infty}$, then the integral $\int_{0}^{.} u_s , \mathrm{d}B_{s}^H$ is Hölder continuous of order $1 - \alpha$ with trajectories in $W_{[0, T]}^{\alpha, \infty}$.
				\begin{theorem}
					Suppose that the coefficients $b(t,x)$ and $\sigma(t,x)$ satisfy the assumptions $(Hb)$ and $(H\sigma)$ with $1-H<\theta \wedge \beta$. Then there exists a stochastic process $X \in C^{\theta \wedge \frac{1}{2}}([-r,T])$ that is a solution to the stochastic equation \eqref{equ}. 
				\end{theorem}
				\begin{proof}
					The existence of a solution follows directly from the deterministic theorem \eqref{Th exi}. Moreover, the solution is $\theta \wedge \frac{1}{2}$-Holder continuous. By following the same steps as in Proposition \eqref{Prop 1}, we obtain
						$$
							\left\Vert X \right\Vert_{\infty,\alpha,T}^{2N} \leq C_{N}\left\lbrace 1+\left(1+\Lambda_{\alpha}(g)^{2N}\right)\int_{0}^{T}\left((T-y)^{-2\alpha}+y^{-\alpha}\right)\left(1+\left\Vert X \right\Vert_{\infty,\alpha,y}^{2N}\right)\mathrm{d}y \right\rbrace, N\geq 1.
						$$ 
					Applying a Gronwall-type lemma \cite[Lemma 6.7]{NR}, it follows that
						$$
								\left\Vert X \right\Vert_{\infty,\alpha,T}^{2N} \leq C_{1}\exp\left(C_{N}\Lambda_{\alpha}(B^{H})^{2N/(1-2\alpha)}\right)<+\infty.
						$$ 
				\end{proof}
				\section*{Acknowledgement}
				The author would like to express his sincere gratitude to Professor Georgiy Shevchenko for suggesting Lemma \eqref{tech-est}, which significantly streamlined the proof of the proposition \eqref{Prop 1}. I would also like to express my deep gratitude to Professor Mohamed Erraoui for their guidance, enthusiastic encouragement and for many stimulating conversations.
	
\end{document}